\documentclass{amsart}

\usepackage[T1]{fontenc}
\usepackage[utf8]{inputenc} 

\usepackage{graphicx}
\usepackage{amsmath}
\usepackage{amsthm}
\usepackage{amsfonts}
\usepackage{mathtools}
\usepackage{amssymb}
\usepackage{hyperref}
\usepackage{caption}
\usepackage{stmaryrd}
\usepackage{enumitem}

\begin{document}

\newcommand{\dc}{\#^\mathrm{dc}}
\newcommand{\ent}{\#^\mathrm{ent}}
\newcommand{\e}{\varepsilon}
\newcommand{\tp}{\mathrm{tp}}
\newcommand{\GH}{\mathrm{GH}}
\newcommand{\Lcal}{\mathcal{L}}
\newcommand{\frk}{\mathfrak}
\newcommand{\ol}{\overline}
\newcommand{\otherpaper}{\,}

\providecommand{\dotdiv}{
  \mathbin{
    \vphantom{+}
    \text{
      \mathsurround=0pt 
      \ooalign{
        \noalign{\kern-.35ex}
        \hidewidth$\smash{\cdot}$\hidewidth\cr 
        \noalign{\kern.35ex}
        $-$\cr 
      }%
    }%
  }%
}

\newcommand*{\approxx}{%
  \mathrel{\vcenter{\offinterlineskip
  \hbox{$\sim$}\vskip-.35ex\hbox{$\sim$}\vskip-.35ex\hbox{$\sim$}}}}

\theoremstyle{plain}
\newtheorem{thm}{Theorem}[section]
\newtheorem{prop}[thm]{Proposition}
\newtheorem{lem}[thm]{Lemma}
\newtheorem{cor}[thm]{Corollary}
\newtheorem{fact}[thm]{Fact}
\newtheorem{ex}[thm]{Example}
\newtheorem{quest}[thm]{Question}

\theoremstyle{definition}
\newtheorem{defn}[thm]{Definition}
\newtheorem{nota}[thm]{Notation}

\newenvironment{desc}{\renewcommand{\proofname}{Description}\begin{proof}}{\end{proof}}

\newtheorem{innercustomgeneric}{\customgenericname}
\providecommand{\customgenericname}{}
\newcommand{\newcustomtheorem}[2]{%
  \newenvironment{#1}[1]
  {%
   \renewcommand\customgenericname{#2}%
   \renewcommand\theinnercustomgeneric{##1}%
   \innercustomgeneric
  }
  {\endinnercustomgeneric}
}

\makeatletter
\DeclareRobustCommand{\cset}{\@ifstar\star@cset\normal@cset}
\newcommand{\star@cset}[1]{\left\llbracket#1\right\rrbracket}
\newcommand{\normal@cset}[2][]{\mathopen{#1\llbracket}#2\mathclose{#1\rrbracket}}
\makeatother

\newcustomtheorem{customthm}{Theorem}
\newcustomtheorem{customlemma}{Lemma}
\newcustomtheorem{customprop}{Proposition}

\title{Approximate Categoricity in Continuous Logic}
\author{James Hanson}
\email{jehanson2@wisc.edu}
\address{Department of Mathematics, University of Wisconsin--Madison, 480 Lincoln Dr., Madison, WI 53706}
\date{\today}

\begin{abstract}
We explore approximate categoricity in the context of distortion systems, introduced in our previous paper \cite{Hansona}, which are a mild generalization of perturbation systems, introduced by Ben Yaacov \cite{OnPert}. We extend Ben Yaacov's Ryll-Nardzewski style characterization of separably approximately categorical theories from the context of perturbation systems to that of distortion systems. We also make progress towards an analog of Morley's theorem for inseparable approximate categoricity, showing that if there is some uncountable cardinal $\kappa$ such that every model of size $\kappa$ is `approximately saturated,' in the appropriate sense, then the same is true for all uncountable cardinalities. Finally we present some examples of these phenomena and highlight an apparent interaction between ordinary separable categoricity and inseparable approximate categoricity.
\end{abstract}

\maketitle
\vspace{-1em}
\section*{Introduction}

This paper is a direct continuation of \cite{Hansona}, in which we generalized perturbation systems, a notion of approximate isomorphism for continuous logic introduced by Ben Yaacov in \cite{OnPert}, to `distortion systems' in order to accommodate the Gromov-Hausdorff and Kadets distances. Ben Yaacov introduces perturbation systems in order to generalize an unpublished result of Henson's,
 specifically a Ryll-Nardzewski type characterization of Banach space theories that are `approximately separably categorical' with regards to the Banach-Mazur distance. Ben Yaacov's formalism requires that approximate isomorphisms be witnessed by uniformly continuous bijections with uniformly continuous inverses, meaning that it cannot accommodate things like the Gromov-Hausdorff and Kadets distances. Our broader formalism also allowed for generalizations of some of the results of \cite{MSA} to perturbation systems, such as explicit Scott sentences axiomatizing the class of Banach spaces with Banach-Mazur distance $0$ to a given Banach space.

In this paper we will first recall the relevant results of \cite{Hansona} as well as develop some additional machinery to deal with types over sets of parameters. In particular we will need to introduce the correct analog of the $d$-metric on types, which is typically distinct from the metric $\delta_\Delta$ induced on type spaces by the distortion system $\Delta$. Then we will develop a generalization of atomic types to this approximate context, which are necessary to state and prove the main results of this paper. This generalization is a specific instance of concepts developed in \cite{BenYaacov2008}. Finally, we will present some examples of theories with various combinations of exact and approximate categoricity, highlighting a gap in the currently known examples. Our explicit examples are in the context of pure metric spaces with Gromov-Hausdorff and Lipschitz distances, but we should note that there is an earlier explicit construction due to Tellez  of a Banach-Mazur-$\omega$-categorical Banach space that is not $\omega$-categorical \cite{NanoThesis}.

The main results of this paper are an extension of Ben Yaacov's separable categoricity result to distortion systems in general and one direction of an approximate Morley's theorem, namely we will define an appropriate notion of `$\Delta$\nobreakdash-saturation' for a distortion system $\Delta$, show that if a theory is $\Delta$\nobreakdash-$\kappa$\nobreakdash-categorical for some uncountable $\kappa$, then every model of it of density character $\kappa$ is $\Delta$-saturated, and then show that if every model of density character $\kappa$ is $\Delta$\nobreakdash-saturated for some uncountable $\kappa$, then the same is true for any uncountable $\lambda$. The difficulty arises in trying to show that two inseparable $\Delta$-saturated structures of the same density character are `almost $\Delta$-approximately isomorphic', where the `almost' is a technical weakening that's only non-trivial in certain poorly behaved `irregular' distortion systems (all of the four motivating examples are regular). The best we seem to be able to get is that they are `potentially almost $\Delta$-approximately isomorphic,' i.e.\ almost $\Delta$-approximately isomorphic in a forcing extension in which they are collapsed to being separable.

For the general formalism of continuous logic and the majority of the notation used here, see \cite{MTFMS}. As opposed to \cite{MTFMS}, however, we (implicitly) opt for the extended definition of formula they allude to in and before Proposition 9.3. Specifically, we allow arbitrary continuous functions from $\mathbb{R}^\omega$ to $\mathbb{R}$ as connectives. This means that formulas may use countably many constants and have countably many free variables, but also, most importantly, that the collection of formulas is closed under uniformly convergent limits up to logical equivalence.  Note that this does not increase the expressiveness of first-order continuous logic, despite the presence of infinitary connectives. Continuous generalizations of $\Lcal_{\omega_1\omega}$ (such as those studied in \cite{MSA}) are fundamentally more expressive because the infinitary connectives introduced there ($\sup$ and $\inf$ of sequences of formulas) are not continuous on $\mathbb{R}^\omega$ (although the resulting formulas in \cite{MSA} are still continuous on structures).

 Here are the rest of the notational conventions used in this paper.

\begin{nota}
\hfill
\begin{itemize}
\item Let $(X,d)$ be a metric space. Let $x\in X$, $A \subseteq X$, and $\varepsilon > 0$.
\begin{enumerate}[label=(\roman*)]
    \item $B_{\leq \varepsilon}^d(x) = \{y \in X : d(x,y) \leq \varepsilon \}$


\item $d(x,A) = \inf\{d(x,y):y \in A\}$



\item $\dc (A,d)$, the metric density character of $A$ with regards to $d$, is the minimum cardinality of a $d$-dense subset of $A$.

\item $\ent_{>\e} (A,d)$, the $\e$-metric entropy of $A$ with regards to $d$, is $\sup|B|$, where $B\subseteq A$ ranges over $({>}\e)$-separated sets.
\end{enumerate}
We will drop the $d$ if the metric is clear from context.

\item To avoid confusion with the established logical roles of $\wedge$ and $\vee$ we will avoid using this notation for $\min$ and $\max$ but in the interest of conciseness we will let $x \uparrow y \coloneqq \max(x,y)$ and $x \downarrow y \coloneqq \min(x,y)$. Note that  $\mathfrak{M} \models \varphi \uparrow \psi \leq 0$ if and only if $\mathfrak{M} \models \varphi \leq 0$ and $\mathfrak{M} \models \psi \leq 0$, and likewise $\mathfrak{M} \models \varphi \downarrow \psi \leq 0$ if and only if $\mathfrak{M} \models \varphi \leq 0$ or $\mathfrak{M} \models \psi \leq 0$.

We take $\uparrow$ and $\downarrow$ to have higher binding precedence than addition but lower binding precedence than multiplication, so for example $ab\uparrow c+d = ((ab)\uparrow c)+d$. We will never write expressions like $x\uparrow a \downarrow b$.


\item If $\varphi$ is a formula and $r$ is a real number (or perhaps another formula), we will write expressions such as $\cset{\varphi < r}$ and $\cset{\varphi \geq r}$ to represent the sets of types (in some type space that will be clear from context) satisfying the given condition.
\end{itemize}

\end{nota}

Here we will fix a few basic concepts from \cite{Hansona}.

\begin{defn}
Fix a language $\Lcal$ with sorts $\mathcal{S}$, $\Lcal$-pre-structures $\frk{M}$ and $\frk{N}$, and tuples $\bar{m} \in \frk{M}$ and $\bar{n} \in \frk{N}$ of the same length with elements in the same sorts.
\begin{itemize}
    \item[(i)]  The \emph{sort-by-sort product of $\frk{M}$ and $\frk{N}$}, written $\frk{M}\times_\mathcal{S}\frk{N}$, is the collection $\bigsqcup_{s\in\mathcal{S}} s(\frk{M})\times s(\frk{N})$. If $\Lcal$ is single-sorted we will take $\times_\mathcal{S}$ to be the ordinary Cartesian product.
    \item[(ii)] A \emph{correlation between $\frk{M}$ and $\frk{N}$} is a set $R \subseteq \frk{M} \times_\mathcal{S} \frk{N}$ such that for each sort $s$,  $R\upharpoonright s \coloneqq R\upharpoonright s(\frk{M})\times s(\frk{N})$ is a total surjective relation. We will write $\mathrm{cor}(\frk{M},\bar{m};\frk{N},\bar{n})$ for the collection of correlations between $\frk{M}$ and $\frk{N}$ such that for each index $i$ less than the length of $\bar{m}$, $(m_i,n_i)\in R$ (for any binary relation we will abbreviate this condition as $(\bar{m},\bar{n})\in R$). If $\bar{m}$ and $\bar{n}$ are empty we will write $\mathrm{cor}(\frk{M},\frk{N})$. 
    \item[(iii)] An \emph{almost correlation between $\frk{M}$ and $\frk{N}$} is a correlation between dense sub-pre-structures of $\frk{M}$ and $\frk{N}$. We will write $\mathrm{acor}(\frk{M},\bar{n};\frk{N},\bar{m})$ for the collection of almost correlations $R$ between $\frk{M}$ and $\frk{N}$ such that $(\bar{m},\bar{n})\in R$.
    
\end{itemize}
\end{defn}

\section{Distortion Systems} \label{sect:dist-sys}

Now we will recall definitions and results from our previous paper \cite{Hansona} that are integral to the results in this paper. To see proofs of these results refer to \cite{Hansona}.

A distortion system is a generalization of the definition of Gromov-Hausdorff distance given in terms of correlations and distortions of correlations, where distortion is defined by
$$\mathrm{dis}(R) = \frac{1}{2}\sup_{(\bar{x},\bar{y}) \in R}\left|d(x_0,x_1) - d(y_0,y_1)\right|,$$
where $R$ is a correlation between some pair of metric spaces $X$ and $Y$. The Gromov-Hausdorff distance between two metric spaces is the infimum of the distortions of correlations between them. This suggests a generalization to languages involving predicate symbols other than just the metric, in which we compute distortion relative to those formulas as well. This approach is highly language dependent and only presents one notion of approximate isomorphism for a given language. A more flexible approach is to allow for an arbitrary collection of formulas in the definition of distortion, yielding this definition.

\begin{defn} \label{defn:main-defn}
A set of formulas $\Delta$ is a \emph{distortion system for $T$} if it is logically complete and closed under renaming variables, quantification, $1$-Lipschitz connectives, logical equivalence modulo $T$, and uniformly convergent limits.

Let $\Delta$ be a distortion system (or any other collection of (finitary) $\Lcal$-formulas) and let $T$ be an $\Lcal$-theory. Let $\frk{M},\frk{N} \models T$ with $\bar{m}\in\frk{M}$ and $\bar{n}\in\frk{N}$.
\begin{itemize}
    \item[(i)] For any relation $R \subseteq \frk{M}\times_\mathcal{S} \frk{N}$, we define the \emph{distortion of $R$ with respect to $\Delta$} as follows:
$$\mathrm{dis}_\Delta(R) =\sup\{|\varphi^\frk{M}(\bar{m})-\varphi^\frk{N}(\bar{n})|:\varphi \in \Delta, (\bar{m},\bar{n})\in R\}$$
    \item[(ii)] We define the \emph{$\Delta$-distance between $(\frk{M},\bar{m})$ and $(\frk{N},\bar{n})$} as follows:
    $$\rho_\Delta(\frk{M},\bar{m};\frk{N},\bar{n})=\inf \{\mathrm{dis}_\Delta(R) :R\in \mathrm{cor}(\frk{M},\bar{m};\frk{N},\bar{n})\}$$
    If $\bar{m}$ and $\bar{n}$ are empty we will just write $\rho_\Delta(\frk{M},\frk{N})$.
   
    \item[(iii)] We say that $(\frk{M},\bar{m})$ and $(\frk{N},\bar{n})$ are \emph{$\Delta$-approximately isomorphic}, written $(\frk{M},\bar{m})\approxx_\Delta (\frk{N},\bar{n})$, if $\rho_\Delta(\frk{M},\bar{m};\frk{N},\bar{n})=0$. 
\end{itemize}
We will also need notions of \emph{almost $\Delta$-similarity}, written $a_\Delta(\frk{M},\bar{m};\frk{N},\bar{n})$, and \emph{almost $\Delta$-approximate isomorphism}, defined analogously to $\rho_\Delta$ and $\Delta$-\hskip0pt approximate isomorphism, respectively, but with almost correlations instead of correlations.
\end{defn}

The closure requirements on distortion systems are natural in the sense that they give a canonical collection of formulas for the corresponding notion of approximate isomorphism. This requirement is harmless, as shown in \cite{Hansona}, in the sense that $\mathrm{dis}_{\Delta}(R) = \mathrm{dis}_{\ol{\Delta}}(R)$ for any collection of formulas $\Delta$ and almost correlation $R$, where $\ol{\Delta}$ is the closure of $\Delta$ under renaming variables, quantification, $1$-Lipschitz connectives, logical equivalence modulo $T$, and uniformly convergent limits. Typically a distortion system is specified in terms of a handful of atomic formulas, as with the Gromov-Hausdorff distance above. The only worry is that given a collection of (atomic) formulas $\Delta$, the closure $\ol{\Delta}$ may fail to be logically complete, as required by the definition of a distortion system. An easy result, given in \cite{Hansona}, is that if $\Delta$ is logically complete for atomic types (or `atomically complete') then $\ol{\Delta}$ is a distortion system.

As it happens this approach is flexible enough to capture other common notions of approximate isomorphism, such as the Lipschitz distance for metric spaces as well as the Banach-Mazur and Kadets distances for Banach spaces and exact isomorphism in any signature. On the other hand, this broader approach also allows for certain pathological behavior, in particular there may be pairs of structures for which $\rho_\Delta(\frk{M},\frk{N}) > a_\Delta(\frk{M},\frk{N})$. All of the motivating examples fall under one of a couple of niceness conditions that prevent this.

\begin{defn} Let $\Delta$ be a distortion system for $T$.
\begin{itemize}
    \item[(i)] We say that $\Delta$ is \emph{regular} if there is an  $\e>0$ such that for any models $\mathfrak{M},\mathfrak{N}\models T$, any almost correlation $R \in \mathrm{acor}(\frk{M},\frk{N})$ with $\mathrm{dis}_\Delta(R) < \e$, and any $\delta> 0$, there exists a correlation $S\in\mathrm{cor}(\frk{M},\frk{N})$ such that $S \supseteq R$ and $\mathrm{dis}_\Delta(S) \leq \mathrm{dis}_\Delta(R) + \delta$.
    \item[(ii)] We say that $\Delta$ is \emph{functional} if there is an $\e > 0$ such that for any models $\mathfrak{M},\mathfrak{N}\models T$ and any closed $R \in \mathrm{acor}(\frk{M},\frk{N})$, if $\mathrm{dis}_\Delta(R) < \e$, then $R$ is the graph of a uniformly continuous bijection between $\frk{M}$ and $\frk{N}$ with uniformly continuous inverse. 
        \item[(iii)] We say that $\Delta$ is \emph{uniformly uniformly continuous} or \emph{u.u.c.}\ if for every $\e>0$ there exists a $\delta>0$ such that for any models $\mathfrak{M},\mathfrak{N}\models T$ and any almost correlation $R \in \mathrm{acor}(\frk{M},\frk{N})$, $\mathrm{dis}_\Delta(R^{< \delta}) \leq \mathrm{dis}_\Delta(R) + \e$, where $R^{< \delta} = \{(a,b) : (\exists(c,d) \in R) d^{\frk{M}}(a,c),d^{\frk{N}}(b,d) < \delta\}$. \qedhere
\end{itemize}
\end{defn}
%
Note that functional and u.u.c.\ both imply regular.

By closing under certain operations on formulas we get a precise syntactic characterization of the metric $\delta_\Delta$ induced on type space by a distortion system.

\begin{defn}
Let $\Delta$ be a distortion system for $T$. For each $\lambda$ and any $p,q\in S_\lambda$, let

$$\delta_\Delta^\lambda(p,q) = \inf\{\rho_\Delta(\frk{M},\bar{m};\frk{N},\bar{n}):\bar{m}\models p,\bar{n}\models q\}.$$

We will typically drop the $\lambda$ when it is clear from context.

\end{defn}

\begin{prop} \label{prop:metric-char} Let $\Delta$ be a distortion system for $T$.
\begin{itemize}
    \item[(i)] $\delta^\lambda_\Delta(p,q)=\sup_{\varphi \in \Delta}|\varphi(p)-\varphi(q)|$, where $\varphi(r)$ means the unique value of $\varphi$ entailed by the type $r$.
    \item[(ii)] $\delta^\lambda_\Delta$ is a topometric on $S_\lambda(T)$, i.e.\ it is lower semi-continuous and refines the topology.
    \item[(iii)] (Monotonicity) For any $p,q\in S_{\lambda+\alpha}(T)$, if $p^\prime, q^\prime \in S_\lambda (T)$ are restrictions of $p$ and $q$ to the first $\lambda$ variables, then $\delta_\Delta^\lambda(p^\prime,q^\prime)\leq \delta_\Delta^{\lambda+\alpha}(p,q)$.
    \item[(iv)] For any $p,q \in S_\lambda(T)$ and any permutation $\sigma: \lambda \rightarrow \lambda$, $d_\Delta^n(p,q) = d_\Delta^n(\sigma p, \sigma q)$, where $\sigma r$ is the type $r(x_{\sigma(0)},x_{\sigma(1)},\dots)$.
    \item[(v)] (Extension) For any $p,q \in S_\lambda(T)$ and $p^\prime \in S_{\lambda + \alpha}(T)$ extending $p$ there exists a $q^\prime \in S_{\lambda + \alpha}(T)$ extending $q$ such that $d_\Delta^\lambda(p,q) = d_\Delta^{\lambda + \alpha}(p^\prime, q^\prime)$.
    \item[(vi)] For any infinite $\lambda$, $\delta^\lambda_\Delta(p,q) = \sup \delta^n_\Delta(p^\prime,q^\prime)$, where $p^\prime$ and $q^\prime$ range over restrictions of $p$ and $q$ to finite tuples of variables.
\end{itemize}
\end{prop}

 Proposition \ref{prop:metric-char} also gives a picture more similar to the definition of perturbation systems, which are presented in terms of a compatible family of topometrics on type spaces. An analogous characterization (essentially the converse of the key points of Proposition \ref{prop:metric-char}) of distortion systems is given in \cite{Hansona}, but it is unneeded here.

Given a correlation $R$ between two metric structures $\frk{M}$ and $\frk{N}$ if we take the metric closure of $R$ we do not increase its distortion, so we can bundle this together as a single metric structure $(\frk{M},\frk{N},R)$, with $R$ encoded by a distance predicate on $\frk{M} \times \frk{N}$ making it into a definable set. Any sufficiently saturated elementary extensions of this structure will give correlations with the same distortion. Anything elementarily equivalent to this structure gives at least an almost correlation. Several other technical considerations necessitate the use of almost correlations in general, but for regular distortion systems it is not necessary to consider almost correlations. All of the motivating examples---the Gromov-Hausdorff and Lipschitz distances for metric spaces and the Banach-Mazur and Kadets distances for Banach spaces---are well behaved in this sense. The collection of structures of the form $(\frk{M},\frk{N},R)$ with $\frk{M},\frk{N}\models T$, $R$ a closed almost correlation, and $\mathrm{dis}_\Delta(R)\leq \e$ turns out to be an elementary class. In \cite{Hansona} we gave a precise characterization of theories arising this way and showed that the distortion system can be recovered from the corresponding family of theories.

\subsection{Parameters in Distortion Systems and $d_\Delta$}

Eventually we will need the correct analog of the $d$-metric for counting types in stability considerations and other things. This concept was introduced by Ben Yaacov in the slightly less general context of perturbations \cite{OnPert}. Our $\delta_\Delta$ is analogous to his $\mathfrak{p}$ and our $d_\Delta$ is analogous to his $\mathfrak{p}^0_{\bar{a}}$.

 Given a complete theory $T$ and a collection of parameters $A$ in some $\frk{M}\models T$, $T_A$ is the theory in the language $\Lcal_A$ with constants added for the parameters $A$. Given a distortion system $\Delta$ for $T$, there is a natural way to extend it to a distortion system $\Delta(A)$ for $T_A$.

\begin{defn}
Let $\Delta$ be a distortion system for a complete theory $T$. Let $A\subseteq \frk{M} \models T$ be some set of parameters. $\Delta(A) = \{\varphi(\bar{x},\bar{a}):\varphi \in \Delta, \bar{a} \in A\}$.
\end{defn}

Clearly $\Delta(A)$ is still a distortion system and it's easy to see that for models $\frk{M}$, $\frk{N}$ containing $A$, we have $\rho_{\Delta(A)}(\frk{M},\frk{N})=\rho_{\Delta}(\frk{M},A;\frk{N},A)$.

\begin{defn}
For any distortion system $\Delta$ and set of fresh constant symbols $C$, let $D(\Delta,C)$ be $\ol{D_0(\Delta,C)}$, where
$D_0(\Delta,C) = \Delta \cup \{d(x,c)\}_{c\in C}.$
\end{defn}



\begin{prop}
If $\Delta$ is a distortion system, then for any set of fresh constant symbols $C$, $D(\Delta,C)$ is a distortion system.
\end{prop}
\begin{proof}
This follows immediately from the fact regarding closures of atomically complete collections of formulas mentioned in Section \ref{sect:dist-sys}. 
\end{proof}

The fact that $D(\Delta,C)$ is a distortion system isn't what is important about it, although it is convenient. What is important is that $\delta^0_{D(\Delta,\bar{c})}$ plays an analogous role to that of the $d$-metric in type spaces. In particular if $\Delta$ is the collection of all formulas and $p,q\in S_n(T)$ are two types, then $\delta^0_{D(\Delta,\bar{c})}(p(\bar{c}),q(\bar{c})) = d(p,q)$ (note that in this expression $p(\bar{c})$ and $q(\bar{c})$ are $0$-types, i.e.\ complete $\Lcal_{\bar{c}}$-theories). To this end we will introduce notation to make the analogy more prominent.

\begin{defn}
If $T$ is a complete theory, $\Delta$ is a distortion system for $T$, and $A$ is some set of parameters in some model of $T$, then for any $\lambda$ and $p,q\in S_\lambda(A)$, we let
$$d_{\Delta,A}(p,q) = \delta^0_{D(\Delta(A),\bar{c})}(p(\bar{c},A),q(\bar{c},A)).$$

We will drop $A$ when it is empty.
\end{defn}

Given how many layers there are to the definition of $d_{\Delta,A}$, the following will be useful for computing estimates of $d_{\Delta,A}$ and is really the best way to think about it. But first recall the following lemma from \cite{Hansona}.

\begin{lem} \label{lem:coarsest}
Let $\Delta$ be a distortion system. For every predicate symbol $P$ and every $\e > 0$ there is a $\delta > 0$ such that if $\rho_\Delta(\frk{M},\bar{m};\frk{N},\bar{n}) < \delta$ then $|P^\frk{M}(\bar{m}) - P^\frk{N}(\bar{n})| < \e$.
\end{lem}

\begin{prop} \label{prop:dd-estimate} Let $T$ be a complete theory, let $\Delta$ be a distortion system for $T$.
\begin{itemize}
    \item[(i)] For every $\e>0$ there is a $\delta>0$ such that if there are models $\frk{M},\frk{N} \models T$ both containing some set of parameters $A$, tuples $\bar{m} \in \frk{M}$ and $\bar{n},\bar{b} \in \frk{N}$ such that $\bar{m}\models p$ and $\bar{n}\models q$, and an $R\in \mathrm{cor}(\frk{M},A\bar{m};\frk{N},A\bar{b})$, and $ \mathrm{dis}_\Delta(R), d^\frk{N}(\bar{n},\bar{b}) \leq \delta$, then $d_{\Delta,A}(p,q) \leq \e$.
    \item[(ii)] For any set of parameters $A$, if $d_{\Delta,A}(p,q) \leq \e$, then there exists models $\frk{M}, \frk{N}\models T$ containing $A$, tuples $\bar{m}\in\frk{M}$ and $\bar{n},\bar{b}\in\frk{N}$ such that $\frk{M}\models p(\bar{m})$ and $\frk{N}\models q(\bar{n})$, and an $R\in\mathrm{cor}(\frk{M},A\bar{m};\frk{N},A\bar{b})$ such that $\mathrm{dis}_\Delta(R) \leq \e$ and $d^\frk{N}(\bar{n},\bar{b})\leq \e$.
\end{itemize}
\end{prop}
\begin{proof} 

\emph{(i):}  Fix $\e > 0$, By Lemma \otherpaper $\ref{lem:coarsest}$, there is a $\delta > 0$ such that if $\delta_\Delta(\mathrm{tp}(ab),
\allowbreak \mathrm{tp}(ce)) \leq \delta$, then $|d(a,b)-d(c,e)|\leq \frac{1}{2}\e$. Without loss assume that $\delta < \frac{1}{2}\e$. Assume that there are models $\frk{M},\frk{N}\models T_A$ with tuples $\bar{m}\in \frk{M}$ and $\bar{n},\bar{b} \in \frk{N}$ such that $\bar{m}\models p$ and $\bar{n}\models q$, and an $R \in \mathrm{cor}(\frk{M},A\bar{m};\frk{N},A\bar{b})$, such that $\mathrm{dis}_\Delta(R) \leq \delta$ and $d^\frk{N}(\bar{n},\bar{b}) \leq \delta$.

We need to compute $\mathrm{dis}_{D(\Delta(A),\bar{c})}(R) = \mathrm{dis}_{D_0(\Delta(A),\bar{c})}(R)$. 
Clearly we already have that $\mathrm{dis}_{D_0(\Delta(A),\bar{c})}(R) \geq \mathrm{dis}_{\Delta(A)}(R)$. We just need to compute $\sup_{c\in \bar{c},(u,v)\in R}|d^{\frk{M}}(u,\allowbreak c)-d^{\frk{N}}(v,c)|$. For any $i<|\bar{m}|$, we have that $(m_i,b_i)\in R$. So for any $(u,v)\in R$, we have that $|d^\frk{M}(m_i,u)-d^\frk{N}(b_i,v)| \leq \frac{1}{2}\e$, so we also have $|d^\frk{M}(m_i,u)-d^\frk{N}(n_i,v)| \leq \frac{1}{2}\e + \delta < \e$. Therefore all together we have that $\mathrm{dis}_{D(\Delta(A),\bar{c})}(R)=\mathrm{dis}_{D_0(\Delta(A),\bar{c})}(R)\leq \e$, as required.

\emph{(ii):} We have that since $d_\Delta(p,q)=\delta^0_{D(\Delta(A),\bar{c})}\leq \e$, we can construct models $(\frk{M},\bar{m})\models p(\bar{m})$ and $(\frk{N},\bar{n})\models q(\bar{n})$ and an $R\in\mathrm{cor}(\frk{M},\bar{m};\frk{N},\bar{n})$ such that $\mathrm{dis}_{D(\Delta(A),\bar{c})}(R) \leq \e$. Find $\bar{b} \in \frk{N}$ such that $(\bar{m},\bar{b})\in R$. Now we have that $|d^\frk{M}(\bar{m},\bar{m})-d^\frk{N}(\bar{b},\bar{n})| \leq \e$, so in particular $d^\frk{N}(\bar{n},\bar{b})\leq \e$, as required.
\end{proof}

In particular $d_{\Delta,A}$ is always uniformly dominated by $d$ on $S_n(A)$ and $\delta_\Delta$ restricted to $S_n(A)$. Moreover as witnessed by the identity correlation on a sufficiently saturated model of $T_A$, $d_{\Delta,A}\leq d$ always holds.

\begin{cor} \label{cor:dd-unif-eq}
If $T$ is a complete theory and $\Delta$ is a u.u.c.\ distortion system for $T$, then for every $\e > 0$ there is a $\delta > 0$ such that for any $\lambda$ and any types $p,q\in S_\lambda(T)$, if $d_\Delta(p,q) < \delta$ then $\delta_\Delta(p,q) < \e$, i.e.\ $\delta_\Delta$ and $d_\Delta$ are uniformly equivalent.
\end{cor}

In fact it's not hard to see that $\delta_\Delta$ and $d_\Delta$ being uniformly equivalent like this (in a way that is uniform across all parameter free type spaces) characterizes u.u.c.\ distortion systems. Furthermore this means that with u.u.c.\ $\Delta$ we don't need to be careful about the distinction between $(S_n(A),d_{\Delta})$ and $(S_n(A),d_{\Delta, A})$, as these two metrics are always uniformly equivalent.




\section{Approximately Atomic Types}

The following concepts, developed in general by Ben Yaacov in \cite{BenYaacov2008}, will be central in this paper, although we have chosen to use the term `$d$-atomic' rather than `$d$\nobreakdash-isolated' to avoid confusion with `isolated with respect to $d$' and to emphasize that these notions will play roles analogous to atomic types.

\begin{defn}
Let $X$ be a topological space and $d:X^2 \rightarrow \mathbb{R}$ a metric (not necessarily related to the topology). Let $x$ be a point in $X$.
\begin{itemize}
    \item $x$ is \emph{$d$-atomic} if $x \in \mathrm{int}B_{\leq \e}(x)$ for every $\e>0$.
    \item $x$ is \emph{weakly $d$-atomic} if $\mathrm{int}B_{\leq \e}(x) \neq \varnothing$ for every $\e>0$.
\end{itemize}
If $Y$ is a subspace containing $x$, and $x$ is (weakly) $d$-atomic in $Y$, we will say that $x$ is \emph{relatively (weakly) $d$-atomic in $Y$} or just \emph{(weakly) $d$-atomic-in-$Y$} (with interiors computed in the subspace topology on $Y$).
\end{defn}

The concept of weak $d$-atomicity will only be important in Subsection \ref{subsec:w-cat} in the context of separable approximate categoricity.

\begin{lem}
If $(X,d)$ is a compact topometric space, then for any closed set $F \subseteq X$ $x \in F$ is $d$-atomic-in-$F$ if and only if there is a continuous function $f:X\rightarrow \mathbb{R}$ such that $f(x)=0$ and for all $y\in F$, $d(x,y) \downarrow 1 \leq f(y)$.
\end{lem}
\begin{proof}
Topometric spaces are always Hausdorff so it is sufficient to construct a continuous function on just $F$ and then we can extend it by the Tietze extension theorem.

Let $G_0 = B_{\leq 1/2}^{d}(x)$ (as a subset of $F$). For every $i<\omega$, let $U_i = \mathrm{int}_F G_i$ (once $G_i$ is defined). Given $G_i$, and therefore also $U_i$, find an $\e_i > 0$ small enough that $B_{\leq \e_i}^d(x) \subseteq U_i$, which always exists by compactness, and such that $\e_i < 2^{-i}$.

For each $i<\omega$, let $f_i:F\rightarrow [0,2^{-i-1}]$ be a continuous function such that $G_{i+1} \subseteq f_i^{-1}(0)$ and $F \setminus U_i \subseteq f_i^{-1}(2^{-i-1})$, which is possible by Urysohn's lemma. Now let $f = \sum_{i<\omega}f_i$. $f$ has the required properties.
\end{proof}

The material after this point will only be important in Subsection \ref{subsec:mor-thm} in the context of inseparable approximate categoricity.

\begin{cor}
Fix a complete first-order theory $T$, distortion system $\Delta$ for $T$, parameter set $A$, and a type $p\in S_n(A)$.
\begin{itemize}
    \item[(i)] $p$ is $\delta_{\Delta(A)}$-atomic if and only if there is an $A$-formula $\varphi$ such that $\varphi(p)=0$ and for any formula $\psi(\bar{x},\bar{y}) \in \Delta$ and $\bar{a}\in A$, $|\psi(q,\bar{a})-\psi(p,\bar{a})|\downarrow 1 \leq \varphi(q)$ for all $q\in S_n(A)$.
    \item[(ii)] $p$ is $d_{\Delta,A}$-atomic if and only if there is an $A$-formula $\varphi$ such that $\varphi(p)=0$ and for any formula $\psi(\bar{c},\bar{y})\in D(\Delta(A),\bar{c})$ and $\bar{a} \in A$, $|\psi(q,\bar{a})-\psi(p,\bar{a})| \downarrow 1\leq \varphi(q)$ for all $q\in S_n(A)$.
\end{itemize}
\end{cor}

We will make use of the following fact.

\begin{fact}[\cite{BenYaacov2008}]
If $(X,d)$ is a topometric space then any continuous function $f:X\rightarrow \mathbb{R}$ is uniformly continuous with regards to $d$.
\end{fact}

\begin{prop} \label{prop:atom-iff}
Fix a complete first-order theory $T$, distortion system $\Delta$ for $T$, parameter set $C$, and tuples $\bar{a}$ and $\bar{b}$ with $|\bar{a}|=n$ and $|\bar{b}|=m$.
\begin{itemize}
    \item[(i)] $\mathrm{tp}(\bar{b}/C)$ is $\delta_{\Delta(C)}$-atomic and $\mathrm{tp}(\bar{a}/C\bar{b})$ is $\delta_{\Delta(C\bar{b})}$-atomic if and only if $\mathrm{tp}(\bar{a}\bar{b}/C)$ is $\delta_{\Delta(C)}$-atomic.
    \item[(ii)] If $\mathrm{tp}(\bar{b}/C)$ is $d_{\Delta,C}$-atomic and $\mathrm{tp}(\bar{a}/C\bar{b})$ is $d_{\Delta,C\bar{b}}$-atomic, then $\mathrm{tp}(\bar{a}\bar{b}/C)$ is $d_{\Delta,C}$-atomic.
    \item[(iii)] If $\mathrm{tp}(\bar{a}\bar{b}/C)$ is $d_{\Delta,C}$-atomic, then $\mathrm{tp}(\bar{b}/C)$ is $d_{\Delta,C}$-atomic.    
    \item[(iv)] If $\Delta$ is u.u.c.\ and $\mathrm{tp}(\bar{a}\bar{b}/C)$ is $d_{\Delta,C}$-atomic, then $\mathrm{tp}(\bar{a}/C\bar{b})$ is $d_{\Delta,C\bar{b}}$\nobreakdash-atomic.
    \item[(v)] If $\mathrm{tp}(\bar{b}/C)$ is $\delta_{\Delta(C)}$-atomic or $d$-atomic, then it is $d_{\Delta,C}$-atomic.
\end{itemize}
\end{prop}

\begin{proof}
\emph{(i):} $(\Rightarrow):$ Let $\varphi(\bar{y})$ be a $C$-formula witnessing that $\mathrm{tp}(\bar{b}/C)$ is $\delta_{\Delta(C)}$\nobreakdash-atomic and let $\psi(\bar{x},\bar{y})$ be a $C$-formula such that $\psi(\bar{x},\bar{b})$ witnesses that $\mathrm{tp}(\bar{a}/C\bar{b})$ is $\delta_{\Delta(C\bar{b})}$-atomic. By the fact above we have that $\psi(\bar{x},\bar{y})$, a function on $S_{n+m}(C)$, is uniformly continuous in the metric $\delta_{\Delta(C)}$. Let $\alpha:\mathbb{R}\rightarrow \mathbb{R}$ be a modulus of uniform continuity for $\psi(\bar{x},\bar{y})$ with regards to $\delta_{\Delta (C)}$, i.e.\ $\alpha$ is a continuous function with $\alpha(0)=0$ such that for any tuples $\bar{c},\bar{e},\bar{u},\bar{v}$, $|\psi(\bar{c},\bar{e})-\psi(\bar{u},\bar{v})|\leq \alpha(\delta_{\Delta(C)}(\mathrm{tp}(\bar{c}\bar{e}/C),\mathrm{tp}(\bar{u}\bar{v}/C)))$. We may assume that $\alpha$ is strictly increasing and in particular invertible.

Consider the formula $\chi(\bar{x},\bar{y})=\psi(\bar{x},\bar{y})+\alpha(\varphi(\bar{y}))$. Pick $\e > 0$ with $\e < 1$ and let $\bar{c},\bar{e}$ be such that $\chi(\bar{c},\bar{e}) < \e$. This implies that $\varphi(\bar{e}) < \alpha^{-1}(\e)$, so  $\delta_{\Delta(C)}(\mathrm{tp}(\bar{e}/C),\allowbreak\mathrm{tp}(\bar{b}/C)) < \alpha^{-1}(\e)$. Let $(\frk{M},\frk{N},R)$ be a structure witnessing this, i.e.\ $R\in\mathrm{cor}(\frk{M},C\bar{e};\allowbreak\frk{N},C\bar{b})$ and $\mathrm{dis}_{\Delta}(R) < \alpha^{-1}(\e)$. Furthermore assume that $\bar{c} \in \frk{M}$ as well and let $\bar{a}^\prime \in \frk{N}$ be such that $(\bar{c},\bar{a}^\prime)\in R$. Now by construction we have that $\delta_{\Delta(C)}(\mathrm{tp}(\bar{e}\bar{c}/C),\allowbreak\mathrm{tp}(\bar{a}^\prime \bar{b}/C)) < \alpha^{-1}(\e)$. This implies that $|\psi(\bar{c},\bar{e})-\psi(\bar{a}^\prime,\bar{b})| < \alpha(\alpha^{-1}(\e))=\e$, so in particular $\psi(\bar{a}^\prime,\bar{b}) < 2\e$ and by construction $\delta_{\Delta(C)}(\mathrm{tp}(\bar{a}^\prime,\bar{b}/C),\mathrm{tp}(\bar{a}\bar{b}/C)) \downarrow 1 < 2\e $. By the triangle inequality this implies that $\delta_{\Delta(C)}(\mathrm{tp}(\bar{c}\bar{e}/C),\mathrm{tp}(\bar{a}\bar{b}/C)) < 2\e + \alpha^{-1}(\e)$. Since $\alpha$ is strictly increasing, $\alpha^{-1}$ is strictly increasing and the function $\e\mapsto 2\e + \alpha^{-1}(e)$ is invertible. Let $f$ be its inverse and now we have that $f(\psi(\bar{x},\bar{y})\downarrow 1)$ is a formula witnessing that $\mathrm{tp}(\bar{a}\bar{b}/C)$ is $\delta_{\Delta(C)}$-atomic.

$(\Leftarrow):$ Let $\varphi(\bar{x},\bar{y})$ be a $C$-formula witnessing that $\mathrm{tp}(\bar{a}\bar{b}/C)$ is $\delta_{\Delta(C)}$-atomic. Consider the formula $\psi(\bar{y})=\inf_{\bar{x}}\varphi(\bar{x},\bar{y})$. To see that this witnesses that $\mathrm{tp}(\bar{b}/C)$ is $\delta_{\Delta(C)}$-atomic, let $\chi(\bar{y})$ be a $\Delta(C)$-formula such that $\models \chi(\bar{b}) \leq 0$. Let $\bar{b}^\prime$ be a tuple such that $\models \psi(\bar{b}^\prime) \leq 0$. In a sufficiently saturated model there exists a tuple $\bar{a}^\prime$ such that $\models \varphi(\bar{a}^\prime\bar{b}^\prime) \leq 0$, so we have that $\bar{b}^\prime \equiv_C \bar{b}$. Now let $\bar{c}$ be any tuple and assume that $\psi(\bar{c}) < \e \leq 1$. That implies that there exists a tuple $\bar{e}$ such that $\varphi(\bar{e},\bar{c}) < \e$, so $\delta_{\Delta(C)}(\mathrm{tp}(\bar{a}\bar{b}/C),\mathrm{tp}(\bar{e}\bar{c}/C)) < \e$. This implies that $\delta_{\Delta(C)}(\mathrm{tp}(\bar{b}/C),\mathrm{tp}(\bar{c}/C)) < \e$ as well, so $\mathrm{tp}(\bar{b}/C)$ is $\delta_{\Delta(C)}$-atomic.

Finally, consider the formula $\varphi(\bar{x},\bar{b})$. This witnesses that $\mathrm{tp}(\bar{a}/C\bar{b})$ is $\delta_{\Delta(C\bar{b})}$-atomic by the previous corollary and since $\Delta(C\bar{b})$ is the same set of formulas as $\{\varphi(\bar{x},\bar{b}):\varphi(\bar{x},\bar{y})\in \Delta(C)\}$.

\emph{(ii):} We will prove a stronger version of this in Lemma \ref{lem:atom-iff-w}.

\emph{(iii):} The proof of this is the same as the proof of the corresponding statement in part \emph{(i)}. 

\emph{(iv:)} If $\Delta$ is u.u.c.\ then $d_{\Delta,C}$ and $\delta_{\Delta(C)}$ are uniformly equivalent and $d_{\Delta,C\bar{b}}$ and $\delta_{\Delta(C\bar{b})}$ are uniformly equivalent, so the statement follows from part \emph{(i)}.

\emph{(v):} This follows from the facts that $d_{\Delta,C} \leq \delta_{\Delta(C)}$ and $d_{\Delta,C}\leq d$ on $S_n(C)$.
\end{proof}

Compare this proposition to these analogous non-approximate statements:
\begin{itemize}
    \item[(i)] $\mathrm{tp}(\bar{a}\bar{b}/C)$ is topologically isolated in $S_{n+m}(C)$ if and only if $\mathrm{tp}(\bar{b}/C)$ is topologically isolated in $S_{m}(C)$ and $\mathrm{tp}(\bar{a}/C\bar{b})$ is topologically isolated in $S_{n}(C\bar{b})$.
    \item[(ii)] If $\mathrm{tp}(\bar{b}/C)$ and $\mathrm{tp}(\bar{a}/C\bar{b})$ are $d$-atomic, then $\mathrm{tp}(\bar{a}\bar{b}/C)$ is $d$-atomic.
    \item[(iii)] If $\mathrm{tp}(\bar{a}\bar{b}/C)$ is $d$-atomic, then $\mathrm{tp}(\bar{b}/C)$ is $d$-atomic.
    \item[(v)] If $\mathrm{tp}(\bar{b}/C)$ is topologically isolated in $S_m(C)$, then it is $d$-atomic.
\end{itemize}

Topological isolation in $S_n(C)$ is of course the same thing as $\delta_{\Delta(C)}$-atomicity when $\Delta$ is the collection of all formulas.

What fails in the statement `$\mathrm{tp}(\bar{a}\bar{b}/C)$ $d_{\Delta,C}$-atomic implies $\mathrm{tp}(\bar{b}/C\bar{a})$ is $d_{\Delta,C\bar{a}}$-atomic' is just the fact that $d_{\Delta,C\bar{a}}$ will in general be larger than $d_{\Delta,C}$ on $S_m(C\bar{a})$. This is the exact same phenomenon that happens with the ordinary $d$-metric.

A slightly beefier version of  Proposition \ref{prop:atom-iff} part \emph{(ii)} will be useful later for `approximately constructible' models. We will prove this for a sequence of singletons for the sake of notational simplicity, but the proof is the same for a sequence of finite tuples.

\begin{lem} \label{lem:atom-iff-w}
Let $T$ be a complete theory, $\Delta$ a distortion system for $T$, $C$ a parameter set, $\bar{a}$ a tuple of elements, and $\{b_i\}_{i<\omega}$ a sequence of elements. If $\mathrm{tp}(b_{<i}/C)$ is $d_{\Delta,C}$-atomic for every $i<\omega$ and $\mathrm{tp}(\bar{a}/Cb_{<\omega})$ is $d_{\Delta,Cb_{<\omega}}$-atomic, then $\mathrm{tp}(\bar{a}b_{<i}/C)$ is $d_{\Delta,C}$-atomic for every $i<\omega$.
\end{lem}

\begin{proof}
For each $i<\omega$, let $\varphi_i(y_0,\dots,y_{i-1})$ be a $[0,1]$-valued $C$-formula witnessing that $\mathrm{tp}(b_{<i}/C)$ is $d_{\Delta,C}$-atomic, and let $\psi(\bar{x},\bar{y})$ be a $C$-formula such that $\psi(\bar{x},b_{<\omega})$ witnesses that $\mathrm{tp}(\bar{a}/Cb_{<\omega})$ is $d_{\Delta,Cb_{<\omega}}$-atomic. Since $\psi(\bar{x},\bar{y})$ can be realized as a uniformly convergent limit of formulas in finitely many variables, we have that $\psi(\bar{x},\bar{y})$ is uniformly continuous with regards to the metric $$d^\dagger(\bar{x}\bar{y},\bar{x}^\prime\bar{y}^\prime)=\sup_{i<\omega} 2^{-i} d_{\Delta,C}(\mathrm{tp}(\bar{x}y_{<i}/C),\mathrm{tp}(\bar{x}^\prime y_{<i}^\prime/C)).$$
Let $\alpha$ be a modulus of uniform continuity for $\psi$ with regards to $d^\dagger$. We may assume that $\alpha$ is strictly increasing and so in particular invertible. Consider the formulas 
$$\chi_j(\bar{x},y_0,\dots,y_{j-1})=\inf_{y_j,y_{j+1},\dots}\psi(\bar{x},\bar{y})+\sup_{i<\omega} 2^{-i} \alpha(\varphi_i(y_{<i}))$$ 
(despite appearances these are first-order formulas). For any $\bar{a}^\prime b^\prime_{<i}$ such that $\models\chi_i(\bar{a}^\prime b^\prime_{<i}) \leq 0$ there is an elementary extension with an $\omega$-tuple $b^\prime_i b^\prime_{i+1}\dots$ such that $\models\psi(\bar{a}^\prime,b_{<\omega}^\prime) \leq 0$ and $\models \varphi_i(b_{<i}) \leq 0$ for each $i<\omega$, so $\bar{a}^\prime b^\prime_{<i} \equiv_C \bar{a} b_{<i}$.

Let $\beta:\mathbb{R}\rightarrow \mathbb{R}$ be a continuous strictly increasing function with $\beta(0)=0$ that witnesses Proposition \otherpaper \ref{prop:dd-estimate} part \emph{(i)} in this situation, i.e.\ for every $\e>0$ if there are models $\frk{M},\frk{N}\models T$ both containing $C$, tuple $\bar{m}\in\frk{M}$ and $\bar{n},\bar{b}\in\frk{N}$, and an $R\in\mathrm{cor}(\frk{M},C\bar{m};\frk{N},C\bar{n})$ with $\mathrm{dis}_\Delta(R),d^\frk{N}(\bar{n},\bar{b})\leq \beta(\e)$, then $d_{\Delta,C}(\mathrm{tp}(\bar{m}/C),\allowbreak\mathrm{tp}(\bar{n}/C))\leq \e$.

Pick $\e > 0$ with $\e \leq 1$ and $j < \omega$ and let $\bar{c}e_{<j}$ be a tuple such that 
$$\models\chi_j(\bar{c},e_{<j})< \frac{1}{2}\beta^{-1}\left(\frac{1}{2}\e\right) \downarrow 2^{-j}\beta^{-1}\left(\alpha^{-1}\left(\frac{1}{2}\beta^{-1}\left(\frac{1}{2}\e\right)\right) \downarrow 2^{-j-1}\e\right).$$

Let $e_j e_{j+1}\dots$ be a tuple witnessing the infimum in $\chi_j(\bar{c},e_{<j})$, so in particular
$$\models \psi(\bar{c},e_{<\omega}) < \frac{1}{2}\beta^{-1}\left(\frac{1}{2}\e\right) \downarrow 2^{-j}\beta^{-1}\left(\alpha^{-1}\left(\frac{1}{2}\beta^{-1}\left(\frac{1}{2}\e\right)\right) \downarrow 2^{-j-1}\e\right).$$

We also have that $\models2^{-j}\varphi_j(e_{<j}) < 2^{-j}\beta^{-1}(\alpha^{-1}(\frac{1}{2}\beta^{-1}(\frac{1}{2}\e)) \downarrow 2^{-j-1}\e),$ so $$d_{\Delta,C}(\mathrm{tp}(e_{<k}/C),\mathrm{tp}(b_{<j}/C)) < \beta^{-1}\left(\alpha^{-1}\left(\frac{1}{2}\beta^{-1}\left(\frac{1}{2}\e\right)\right) \downarrow 2^{-j-1}\e\right).$$
By Proposition \otherpaper \ref{prop:dd-estimate} part \emph{(ii)} we can find a triple $(\frk{M},\frk{N},R)$ such that $e_{<j} \in \frk{M}$, $b_{<j}\in \frk{N}$, $R\in\mathrm{cor}(\frk{M},Ce_{<j};\frk{N},C\bar{n})$, and there is a tuple $\bar{n}\in\frk{N}$ such that $d^\frk{N}(\bar{n},b_{<j})$ and $\mathrm{dis}_{\Delta}(R)$ are both less than  $\beta^{-1}(\alpha^{-1}(\frac{1}{2}\beta^{-1}(\frac{1}{2}\e)) \downarrow 2^{-j-1}\e)$. We may assume that all of $b_{<\omega}$ is in $\frk{N}$.

If we let $\bar{a}^\prime \in \frk{N}$ be a tuple such that $(\bar{e},\bar{a}^\prime)\in R$, then we have that $$d_{\Delta,C}(\mathrm{tp}(\bar{c}e_{<k}/C),\mathrm{tp}(\bar{a}^\prime b_{<k}/C)) < \alpha^{-1}\left(\frac{1}{2}\beta^{-1}\left(\frac{1}{2}\e\right)\right) \downarrow 2^{-j-1}\e,$$ for each $k\leq j$, by Proposition \otherpaper \ref{prop:dd-estimate} part \emph{(i)}. By construction this implies that $$d^\dagger(\bar{c}e_{<\omega},\bar{a}^\prime b_{<\omega}) < \alpha^{-1}\left(\frac{1}{2}\beta^{-1}\left(\frac{1}{2}\e\right)\right) \downarrow 2^{-j-1}\e.$$ Now by the choice of $\alpha$ we get that $|\psi(\bar{c},e_{<\omega})-\psi(\bar{a}^\prime,b_{<\omega})| < \frac{1}{2}\beta^{-1}(\frac{1}{2}\e)$ and in particular $\psi(\bar{a}^\prime,b_{<\omega}) < \beta^{-1}(\frac{1}{2}\e)$, since $\psi(\bar{c},e_{<\omega}) < \frac{1}{2}\beta^{-1}(\frac{1}{2}\e)$.

This implies that $d_{\Delta,Cb_{<\omega}}(\mathrm{tp}(\bar{a}^\prime/Cb_{<\omega}),\mathrm{tp}(\bar{a}/Cb_{<\omega})) < \beta^{-1}(\frac{1}{2}\e)$.  Proposition \otherpaper \ref{prop:dd-estimate} now implies that $d_{\Delta,C}(\mathrm{tp}(\bar{a}^\prime b_{<j}/C),\mathrm{tp}(\bar{a} b_{<j}/C)) < \frac{1}{2}\e.$

Since $d^\dagger(\bar{c}e_{<\omega},\bar{a}^\prime b_{<\omega}) < 2^{-j-1}\e$, we have by definition that $$2^{-j} d_{\Delta,C}(\mathrm{tp}(\bar{c}e_{<j}/C),\mathrm{tp}(\bar{a}^\prime b_{<j}/C)) < 2^{-j-1}\e$$ 
and so $d_{\Delta,C}(\mathrm{tp}(\bar{c}e_{<j}/C),\mathrm{tp}(\bar{a}^\prime b_{<j}/C)) < \frac{1}{2}\e$. 

By the triangle inequality we have that $d_{\Delta,C}(\mathrm{tp}(\bar{c}e_{<j}/C),\mathrm{tp}(\bar{a} b_{<j}/C)) < \frac{1}{2}\e + \frac{1}{2}\e = \e$. Since we can do this for any $\e > 0$ and $j<\omega$, we have that $\mathrm{tp}(\bar{a}b_{<j}/C)$ is $d_{\Delta,C}$-atomic for every $j<\omega$.
\end{proof}

\section{Approximate Categoricity}

Given an approximate notion of isomorphism it is natural to consider a corresponding notion of approximate categoricity.

\begin{defn}
Fix a theory $T$ and a distortion system $\Delta$ for $T$. 
\begin{itemize}
    \item $T$ is \emph{$\Delta$-$\kappa$-categorical} if for any two models $\frk{M},\frk{N}\models T$ of density character $\kappa$, $a_\Delta(\frk{M},\frk{N})=0$.
    \item $T$ is \emph{strongly $\Delta$-$\kappa$-categorical} if for any two models $\frk{M},\frk{N}\models T$ of density character $\kappa$, $\frk{M}\approxx_\Delta \frk{N}$. \qedhere
\end{itemize}
\end{defn}

Of course $\Delta$-$\kappa$-categoricity and strong $\Delta$-$\kappa$-categoricity are equivalent when $\Delta$ is regular. The motivating examples of distortion systems are regular and certainly strong $\Delta$-$\kappa$-categoricity is the more compelling notion, but all of the results in this paper easily generalize to what we are calling $\Delta$-$\kappa$-categoricity, with no apparent gain from assuming strong $\Delta$-$\kappa$-categoricity and no clear way to characterize strong $\Delta$-$\kappa$-categoricity. This, together with the necessity of introducing a weaker version of $\Delta$-$\kappa$-categoricity in Section \ref{subsec:mor-thm}, lead us to this naming convention.

Of course a natural question is whether or not this distinction even matters.

\begin{quest}
Does there exist a theory $T$, a distortion system $\Delta$ for $T$, and an infinite cardinal $\kappa$ such that $T$ is $\Delta$-$\kappa$-categorical but not strongly $\Delta$\nobreakdash-$\kappa$\nobreakdash-\hskip0pt categorical?
\end{quest}


\subsection{Separable Categoricity}\label{subsec:w-cat} 

We can restate Ben Yaacov's generalization of the Ryll-Nardzewski theorem to the context of perturbations in our language. In this section we will extend this result to distortion systems in general.

\begin{thm}[Ben Yaacov \cite{OnPert}] \label{thm:BY-cat}
Let $T$ be a complete theory with non-compact models in a countable language and $\Delta$ a functional distortion system for $T$.
\begin{itemize}
    \item The following are equivalent.
    \begin{itemize}
    \item $T$ is $\Delta$-$\omega$-categorical.
    \item For every finite tuple of parameters $\bar{a}$ and $n<\omega$, every type in $S_n(\bar{a})$ is weakly $d_\Delta$-atomic-in-$S_n(\bar{a})$.
    \item The same but with $n$ restricted to $1$.
    \end{itemize}
    \item If $(S_n(T),d_\Delta)$ is metrically compact for every $n<\omega$ (equivalently if every $\varnothing$-type is $d_\Delta$-isolated), then $T$ is $\Delta$-$\omega$-categorical.
\end{itemize}
\end{thm}

$d_\Delta$ on $S_1(\bar{a})$ is the restriction of $d_\Delta$ on $S_{1+|\bar{a}|}(T)$ to the subspace corresponding to $S_1(\bar{a})$. This is $\tilde{d}_{\frk{p}}$ in Ben Yaacov's notation. Note that $d_\Delta$ on $S_1(\bar{a})$ is not the same thing as $d_{\Delta,\bar{a}}$ on $S_1(\bar{a})$. For example, if $\Delta$ is the collection of all formulas, so that $\rho_\Delta$ corresponds to isomorphism, then $d_\Delta = d$ for the type spaces $S_n(T)$, but over parameters the topometric space $(S_n(\bar{a}),d)$, where $d$ is the ordinary $d$-metric on $S_n(\bar{a})$, is not the same topometric space as $(S_n(\bar{a}),d_\Delta)=(S_n(\bar{a}),d^{S_{n+|\bar{a}|}(T)})$ in general, where $d^{S_{n+|\bar{a}|}(T)}$ is the $d$-metric on $S_{n+|\bar{a}|}(T)$.






\begin{lem}
Fix a countable first-order theory $T$ and a distortion system $\Delta$ for $T$. For any type $p(\bar{x})\in S_n(T)$ for some $n$, and any extension $q(\bar{x},\bar{y})\in S_{n+m}(T)$, there is a separable model $\frk{M}\models T$ such that $\frk{M}$ realizes $p$ and for any $n$-tuple $\bar{a}\in\frk{M}$ and any $\e>0$ there is an $m$-tuple $\bar{b}\in\frk{M}$ such that $d_\Delta(q,\mathrm{tp}(\bar{a}\bar{b})) < d_\Delta(p,\mathrm{tp}(\bar{a}))+\e$.
\end{lem}

\begin{proof}
Let $\frk{M}_0$ be a countable pre-model of $T$ realizing $p$. Proceed inductively. At stage $i$, given $\frk{M}_i$, a countable pre-model of $T$, find $\frk{M}_{i+1}\succeq \frk{M}_i$ such that for every $n$-tuple $\bar{a}\in\frk{M}_i$ there is an $m$-tuple $\bar{b}\in\frk{M}_{i+1}$ such that $d_\Delta(q,\mathrm{tp}(\bar{a}\bar{b})) < d_\Delta(p,\mathrm{tp}(\bar{a})) + 2^{-i}$. Note that this is always possible since $d_\Delta$ has the extension property.

Now let $\frk{M}$ be the completion of the union $\bigcup_{i<\omega}\frk{M}_i$. Let $\bar{a}$ be an $n$-tuple in $\frk{M}$. For any $\e > 0$, find $\bar{a}_i \in \frk{M}_i$ such that $d(\bar{a},\bar{a}_i) < \frac{1}{3} \e$ and such that $2^{-i} < \frac{1}{3}\e$. Then by construction there is $\bar{b} \in \frk{M}$ such that $d_\Delta(q,\mathrm{tp}(\bar{a}_i\bar{b})) < d_\Delta(p,\mathrm{tp}(\bar{a}_i)) + \frac{1}{3}\e$. This implies that $d_\Delta(q,\mathrm{tp}(\bar{a}_i\bar{b})) < d_\Delta(p,\mathrm{tp}(\bar{a})) + \frac{2}{3}\e$, but we also have that $d_\Delta(q,\mathrm{tp}(\bar{a}\bar{b})) < d_\Delta(q,\mathrm{tp}(\bar{a}_i\bar{b}))+ d(\bar{a}_i\bar{b},\bar{a}\bar{b})$. Finally $d(\bar{a}_i\bar{b},\bar{a}\bar{b}) = d(\bar{a}_i,\bar{a}) < \frac{1}{3} \e$, so putting this together gets $d_\Delta(q,\mathrm{tp}(\bar{a}\bar{b})) < d_\Delta(q,\mathrm{tp}(\bar{a}_i\bar{b})) + \e $, as required.
\end{proof}

\begin{defn}
A structure $\frk{M}$ is \emph{approximately $\Delta$-$\omega$-saturated} if for every $\bar{a} \in \frk{M}$, every $p\in S_n(\bar{a})$, and every $\e>0$, there is $\bar{b} \in \frk{M}$ such that $d_\Delta(p,\mathrm{tp}(\bar{a}\bar{b})) < \e$.
\end{defn}

When $\Delta$ is the collection of all formulas, a structure is approximately $\Delta$\nobreakdash-$\omega$\nobreakdash-\hskip0pt saturated if and only if it is approximately $\omega$-saturated, hence the redundant sounding name.

\begin{prop} \label{prop:one-enough}
A structure $\frk{M}$ is approximately $\Delta$-$\omega$-saturated if and only if it is for $1$-types, i.e.\ for every $\bar{a} \in \frk{M}$, every $p\in S_1(\bar{a})$, and every $\e>0$, there is $b \in \frk{M}$ such that $d_\Delta(p,\mathrm{tp}(\bar{a}b)) < \e$.
\end{prop}
\begin{proof}
The $\Rightarrow$ direction is obvious, so we only need to show that if $\frk{M}$ is $\Delta$\nobreakdash-$\omega$\nobreakdash-\hskip0pt saturated for $1$-types, then it is $\Delta$-$\omega$-saturated.

Let $\bar{a}\in \frk{M}$ be a tuple and let $p\in S_n(\bar{a})$ be some type. Pick $\e > 0$. For each $i$ with $0<i \leq n$, let $p_i$ be the restriction of $p$ to the first $i$ variables. 

First find $b_1 \in \frk{M}$ such that $d_\Delta(p_1, \mathrm{tp}(\bar{a}b_1)) < \frac{\e}{n}$.

Now at any stage $i \geq 1$, given $b_1\dots b_i$, find $q_{i+1} \in S_{1}(\bar{a}b_1\dots b_i)$ such that $d_\Delta(p_{i+1},q_{i+1}) = d_\Delta(p_i, \mathrm{tp}(\bar{a}b_1\dots b_i))$.
\end{proof}

\begin{lem} \label{lem:sat-alt}
If $\frk{M}$ is approximately $\Delta$-$\omega$-saturated, then for every tuple $\bar{a}$ and every type $p(\bar{a},\bar{y})\in S_n(\bar{a})$ and $\e > 0$ there is $\bar{b}\bar{c} \in \frk{M}$ such that $d(\bar{a},\bar{b})<\e$ and $\delta_\Delta(p,\mathrm{tp}(\bar{b}\bar{c}))<\e$. 
\end{lem}

\begin{proof}
Pick $\e>0$. Let $\bar{b}_{0}=\bar{a}$
and find $\bar{c}_{0}$ such that $d_{\Delta}(p,\mathrm{tp}(\bar{b}_{0}\bar{c}_{0}))<\frac{1}{2}\e$.
By Proposition  \otherpaper \ref{prop:dd-estimate} there exists a type $q_{0}$ such that
$\delta_{\Delta}(p,q_{0})<\frac{1}{2}\e$ and $d(q_{0},\allowbreak\mathrm{tp}(\bar{b}_{0}\bar{c}_{0}))<\frac{1}{2}\e$.
Let $r_{0}$ be a completion of the type $\{d(\bar{b}_{0}\bar{c}_{0},\bar{x}\bar{y})<\frac{1}{2}\e,q_{0}(\bar{x},\bar{y})\}$
(i.e.\ a type in twice as many variables).

Now at stage $i$, given $\bar{b}_{i}\bar{c}_{i}$ and $r_{i}$,
find $\bar{b}_{i+1}\bar{c}_{i+1}$ such that $$d_{\Delta}(r_{i},\allowbreak\mathrm{tp}(\bar{b}_{i}\bar{c}_{i}\bar{b}_{i+1}\bar{c}_{i+1}))<2^{-i-1}\e.$$
So in particular $d_{\Delta}(q_{i},\mathrm{tp}(\bar{b}_{i+1}\bar{c}_{i+1}))<2^{-i-1}\e$.
By Proposition \otherpaper  \ref{prop:dd-estimate} there exists a type $q_{i+1}$ such
that $\delta_{\Delta}(q_{i},q_{i+1})<2^{-i-1}\e$ and $d(q_{i+1},\allowbreak\mathrm{tp}(\bar{b}_{i+1}\bar{c}_{i+1}))<2^{-i-1}\e$.
Let $r_{i+1}$ be a completion of the type $\{d(\bar{b}_{i+1}\bar{c}_{i+1},\allowbreak\bar{x}\bar{y})<2^{-i-1}\e,q_{i+1}(\bar{x},\bar{y})\}$.

Now by construction $\{\bar{b}_{i}\bar{c}_{i}\}_{i<\omega}$
is a Cauchy sequence in $d$. Let it limit to $\bar{b}\bar{c}$.
By construction we have that $d(\bar{a},\bar{b})<\e$.
$\{q_{i}\}_{i<\omega}$ is also a Cauchy sequence in $\delta_{\Delta}$.
Let it limit to $q$. By construction we have that $\delta_{\Delta}(p,q)<\e$.
Furthermore since $d(q_{i+1},\mathrm{tp}(\bar{b}_{i+1}\bar{c}_{i+1}))<2^{-i-1}\e$,
we have that $\bar{b}\bar{c}\models q$, so in particular
$\delta_{\Delta}(p,\mathrm{tp}(\bar{b}\bar{c}))<\e$,
as required.
\end{proof}

\begin{prop} \label{prop:sat-approx}
 For any countable theory $T$ (not necessarily complete) and distortion system $\Delta$ for $T$, if $\frk{M},\frk{N}\models T$ are separable models that are both approximately $\Delta$-$\omega$-saturated, then $a_\Delta(\frk{M},\frk{N})$ is as small as possible, i.e.\ $a_\Delta(\frk{M},\allowbreak\frk{N})=\delta_\Delta(\mathrm{Th}(\frk{M}),\mathrm{Th}(\frk{N}))$. In particular if $\frk{M} \equiv \frk{N}$ then $a_\Delta(\frk{M},\frk{N})=0$.
\end{prop}

\begin{proof}
Pick $\e>\delta_\Delta(\mathrm{Th}(\frk{M}),\mathrm{Th}(\frk{N}))$. Let $\{m_{2i}\}_{i<\omega}$ be an enumeration of a tail-dense (i.e.\ every final segment is dense) sequence in $\frk{M}$ and let $\{n_{2i+1}\}_{i<\omega}$ be a tail-dense sequence in $\frk{N}$. We will construct an almost correlation between $\frk{M}$ and $\frk{N}$ with a back-and-forth argument with the typical continuous modification that the sequence built will need to `slide around' a little at each stage to make things line up.

$\{a_i^j\}_{j\leq i < \omega}$ will be an array of elements of $\frk{M}$ and $\{b_i^j\}_{j\leq i < \omega}$ will be an array of elements of $\frk{N}$ chosen so that for each fixed $i$, $a_i^j$ and $b_i^j$ are Cauchy sequences in $j$. Their limits, $a_i^\omega$ and $b_i^\omega$, will be the desired correlation with distortion $\leq \e$.

On stage $0$, let $a_0^0 = m_0$ and find $b_0^0$ such that $$\delta_\Delta(\mathrm{tp}(a_0^0),\mathrm{tp}(b_0^0)) = \delta_\Delta(\mathrm{Th}(\frk{M}),\mathrm{Th}(\frk{N})) < \e.$$ 
On odd stage $2k+1$, let $b_{i}^{2k+1} = b_{i}^{2k}$ for $i<2k+1$ and let $b_{2k+1}^{2k+1} = n_{2k+1}$. By the induction hypothesis, $\delta_\Delta(\mathrm{tp}(a_{\leq 2k}^{2k}),\mathrm{tp}(b_{\leq 2k}^{2k})) < \e$. Let $p_{2k+1}$ be an extension of the type $\mathrm{tp}(a_{\leq 2k}^{2k})$ such that $\delta_\Delta (\mathrm{tp}(a_{\leq 2k}^{2k}),\mathrm{tp}(b_{\leq 2k}^{2k}))=\delta_\Delta(p_{2k+1},\mathrm{tp}(b_{\leq 2k+1}^{2k+1}))$. Now by Lemma \ref{lem:sat-alt} we can find $a_{\leq 2k}^{2k+1}$ and $a_{2k+1}^{2k+1}$ such that $d^\frk{M}(a_{\leq 2k}^{2k},a_{\leq 2k}^{2k+1}) < 2^{-k}$ and with $\delta_\Delta(p_{2k+1},\mathrm{tp}(a_{\leq 2k+1}^{2k+1}))$ small enough that $\delta_\Delta(\mathrm{tp}(a_{\leq 2k+1}^{2k+1}),\allowbreak \mathrm{tp}(b_{\leq 2k+1}^{2k+1})) \allowbreak < \e$.

On even stage $2k+2$ we do the same with the roles reversed.

Now clearly for each fixed $i$, $a_i^j$ and $b_i^j$ are Cauchy sequences in $j$, so let $a_i^\omega$ and $b_i^\omega$ be their limits. Note that $d^\frk{M}(a_{2k}^\omega, m_{2k}) \leq 2^{-2k+1}$ and similarly for $b_{2k+1}^\omega$ and $n_{2k+1}$, so we have that $\{a_i^\omega\}_{i<\omega}$ is dense in $\frk{M}$ and $\{b_i^\omega\}_{i<\omega}$ is dense in $\frk{N}$ by the tail-density of the sequences $\{m_{2i}\}$ and $\{n_{2i+1}\}$. So $R=\{(a_i^\omega,b_i^\omega):i<\omega\}$ is an almost correlation between $\frk{M}$ and $\frk{N}$.

By induction we have for each $j<\omega$ that $\delta_\Delta(\mathrm{tp}(a_{\leq j}^k),\mathrm{tp}(b_{\leq j}^k)) < \e$ for all $k$ such that this quantity is defined. By lower semi-continuity of $\delta_\Delta$ this implies that $\delta_\Delta(\mathrm{tp}(a_{\leq j}^\omega),\mathrm{tp}(b_{\leq j}^\omega)) \leq \e$ for all $j<\omega$. Therefore we have that $\mathrm{dis}_\Delta(R) \leq \e$ as well.

Since we can do this for any $\e>\delta_\Delta(\mathrm{Th}(\frk{M}),\mathrm{Th}(\frk{N}))$, we have that $a_\Delta(\frk{M},\allowbreak\frk{N})\leq \delta_\Delta(\mathrm{Th}(\frk{M}),\allowbreak\mathrm{Th}(\frk{N}))$.  Since $a_\Delta(\frk{M},\frk{N}) \geq \delta_\Delta(\mathrm{Th}(\frk{M}),\mathrm{Th}(\frk{N}))$  always holds, we have $a_\Delta(\frk{M},\frk{N}) = \delta_\Delta(\mathrm{Th}(\frk{M}),\mathrm{Th}(\frk{N}))$, as required.
\end{proof}

\begin{cor}
If $\Delta$ is a regular distortion system and $\frk{M},\frk{N}\models T$ are approximately $\Delta$-$\omega$-saturated separable models then $\rho_\Delta(\frk{M},\frk{N}) = \delta_\Delta(\mathrm{Th}(\frk{M}),\mathrm{Th}(\frk{N}))$. In particular if $\frk{M}\equiv\frk{N}$, then $\frk{M} \approxx_\Delta \frk{N}$.
\end{cor}

\begin{prop}
For any countable complete theory $T$ and distortion system $\Delta$ for $T$, if $T$ is $\Delta$-$\omega$-categorical, then every separable model of $T$ is approximately $\Delta$-$\omega$-saturated.
\end{prop}
\begin{proof}
Fix a separable $\frk{M}\models T$. Fix $\bar{a} \in \frk{M}$ and $p(\bar{x},\bar{a})\in S_m(\bar{a})$. Let $\frk{N}$ be the model for $\mathrm{tp}(\bar{a})$ and $p$ guaranteed by the previous lemma. Pick $\gamma > 0$. Find $\delta > 0$ according to Proposition \otherpaper \ref{prop:dd-estimate} \emph{(ii)} with $\e < \frac{1}{4} \gamma$. Without loss assume that $\delta < \frac{1}{4}\gamma$. Let $R\in \mathrm{acor}(\frk{M},\frk{N})$ be closed and have $\mathrm{dis}_\Delta (R) < \delta$. Find $\bar{a}^\prime$ in the domain of $R$ such that $d^\frk{M}(\bar{a},\bar{a}^\prime) < \delta$ and let $\bar{b}\in \frk{N}$ be such that $(\bar{a}^\prime,\bar{b})\in R$. By passing to an $\aleph_1$-saturated elementary extension of $(\frk{M},\frk{N},R)$ and by using Proposition \otherpaper \ref{prop:dd-estimate} we get that $d_\Delta(\mathrm{tp}(\bar{a}),\mathrm{tp}(\bar{b})) \leq \e < \frac{1}{4}\gamma$. 

This implies that there exists $\bar{c} \in \frk{N}$ such that $d_\Delta(p,\mathrm{tp}(\bar{b}\bar{c})) < d_\Delta(\mathrm{tp}(\bar{a}),\allowbreak\mathrm{tp}(\bar{b})) + \frac{1}{4}\gamma$. Now find $\bar{c}^\prime \in \frk{N}$ in the range of $R$ such that $d^\frk{N}(\bar{c},\bar{c}^\prime) <\frac{1}{4} \gamma$. Now we have that $d_\Delta(p,\mathrm{tp}(\bar{b}\bar{c}^\prime)) < d_\Delta(p,\mathrm{tp}(\bar{b}\bar{c})) + d(\bar{c},\bar{c}^\prime)$ and thus $d_\Delta(p,\mathrm{tp}(\bar{b}\bar{c}^\prime)) < d_\Delta(p,\mathrm{tp}(\bar{b}\bar{c})) + \frac{1}{4}\gamma$. Putting this all together gives that $d_\Delta(p,\mathrm{tp}(\bar{b}\bar{c}^\prime)) < \frac{3}{4}\gamma$. 

Now find $\bar{e} \in \frk{M}$ such that $(\bar{e},\bar{c}^\prime) \in R$, so in particular we have that $(\bar{a}^\prime \bar{e},\bar{b}\bar{c}^\prime) \in R$. By passing to an $\aleph_1$-saturated elementary extension of $(\frk{M},\frk{N},R)$ and using Proposition \otherpaper \ref{prop:dd-estimate} again, we have that $d_\Delta(\mathrm{tp}(\bar{a}\bar{e}),\mathrm{tp}\bar{b}\bar{c}^\prime) \leq \e < \frac{1}{4}\gamma$. By the triangle inequality this gives that $d_\Delta(p,\mathrm{tp}(\bar{a}\bar{e})) < \frac{1}{4}\gamma + \frac{3}{4}\gamma = \gamma$.

Since we can do this for any separable $\frk{M}\models T$, any $\bar{a}$, any $p\in S_m(\bar{a})$, and any $\gamma >0$, we have that every separable models of $T$ is approximately $\Delta$-$\omega$-saturated.
\end{proof}

\begin{cor}
A countable theory $T$ with distortion system $\Delta$ is $\Delta$-$\omega$-\hskip0pt categorical if and only if every separable model is approximately $\Delta$-$\omega$-saturated.
\end{cor}

\begin{prop} \label{prop:omit-types}
Fix a countable complete theory $T$ and a distortion system $\Delta$ for $T$. For any  parameters $\bar{a}$ and a type $p(\bar{x},\bar{a})\in S_n(\bar{a})$ the following are equivalent:
\begin{itemize}
    \item For every $\frk{M} \models T$ containing $\bar{b}\equiv \bar{a}$ and for every $\e > 0$, there is $\bar{c} \in \frk{M}$ such that $d_\Delta(p(\bar{x},\bar{b}),\mathrm{tp}(\bar{c}\bar{b})) < \e$.
    \item $p$ is weakly $d_\Delta$-atomic-in-$S_n(\bar{a})$.
\end{itemize}
\end{prop}
\begin{proof}
The second bullet point clearly implies the first bullet, since if $p$ is weakly $d_\Delta$-atomic-in-$S_n(\bar{a})$, then for every $\e>0$, $\mathrm{int}_{S_n(\bar{a})}B^{d_\Delta}_{ \leq \e}(p)$ is non-empty and non-empty open subsets of type space are always realized.

So assume that the second bullet point fails. This implies that there is an $\e > 0$ such that $\mathrm{int}_{S_n(\bar{a})}B^{d_\Delta}_{ \leq \e}(p)= \varnothing$. $\mathrm{int}_{S_n(\bar{a})}B^{d_\Delta}_{ \leq \e}(p)$ is a closed set, so we can build a pre-model $\frk{M}_0\ni \bar{a}$ omitting it. Passing to the completion $\frk{M} = \ol{\frk{M}_0}$, note that any type realized in $\frk{M}$ must be in the metric closure under the ordinary $d$-metric of the set of types realized in $\frk{M}_0$, but since $d \geq d_\Delta$ this implies that they're in the metric closure under $d_\Delta$ of the set of types realized in $\frk{M}_0$, so we have that any type $q\in S_n(\bar{a})$ realized in $\frk{M}$ must have $d_\Delta(p,q) \geq \e$, contradicting the first  bullet point. 
\end{proof}


\begin{thm}\label{thm:w-cat-main}
For any countable complete theory $T$ and distortion system $\Delta$ for $T$,

\begin{itemize}

\item[(i)] $T$ is $\Delta$-$\omega$-categorical if and only if for every tuple of parameters $\bar{a}$ and every $n<\omega$, every $p\in S_n(\bar{a})$ is weakly $d_\Delta$-atomic-in-$S_n(\bar{a})$.
\item[(ii)] Same as the previous statement but only considering types in $S_1(\bar{a})$.
\item[(iii)] If every $S_n(T)$ is metrically compact relative to $d_\Delta$, then $T$ is $\Delta$-$\omega$-\hskip0pt categorical.
\end{itemize}
\end{thm}
\begin{proof}
\textit{(i)}: If $T$ is $\Delta$-$\omega$-categorical then for any finite tuple $\bar{a}$, any type $p(\bar{x},\bar{a})$, and any separable model $\frk{M} \ni \bar{a}$, the condition in the first bullet point of Proposition \ref{prop:omit-types} holds (since $\frk{M}$ is approximately $\Delta$-$\omega$-saturated), i.e.\ the type must be `approximately realized.' Therefore $p$ must be weakly $d_\Delta$\nobreakdash-atomic\nobreakdash-in\nobreakdash-$S_n(\bar{a})$.

Conversely, if for every $\bar{a}$ and $p\in S_n(\bar{a})$, $p$ is weakly $d_\Delta$-atomic-in-$S_n(\bar{a})$, then by Proposition \ref{prop:omit-types} no type over a tuple of parameters can be `approximately omitted,' i.e.\ the first bullet point in Proposition \ref{prop:omit-types} always holds. This implies that every separable model of $T$ is approximately $\Delta$-$\omega$-saturated, therefore $T$ is $\Delta$-$\omega$-categorical.

\textit{(ii):} Clearly if every $p \in S_n(\bar{a})$ is weakly $d_\Delta$-atomic-in-$S_n(\bar{a})$, then the same is true restricting $n$ to $1$.

So assume that for every finite tuple of parameters $\bar{a}$, every $p\in S_1(\bar{a})$ is weakly $d_\Delta$-atomic-in-$S_1(\bar{a})$. It's clear that the proof of Proposition \ref{prop:sat-approx} only requires approximate $\Delta$-$\omega$-saturation for $1$-types, so we get that for any two separable models $\frk{M},\frk{N}\models t$, $a_\Delta(\frk{M},\frk{N})=0$, i.e.\ $T$ is $\Delta$-$\omega$-categorical.

\textit{(iii):} This is enough to imply that every $S_n(\bar{a})$ is metrically compact with regards to $d_\Delta$ as well (since $d_\Delta$ is just the restriction to $S_n(\bar{a})$ as a subspace of $S_{n+|\bar{a}|}(T)$). Therefore every $p\in S_n(\bar{a})$ is $d_\Delta$-atomic, and therefore in particular weakly $d_\Delta$-atomic.
\end{proof}

\begin{cor}
If $T$ is a countable theory, $\Delta$ is a distortion system for $T$, $\bar{a}$ is any tuple of parameters, and $T$ is $\Delta$-$\omega$-categorical, then $T_{\bar{a}}$ is $D(\Delta,\bar{a})$\nobreakdash-$\omega$\nobreakdash-\hskip0pt categorical, i.e.\ for any two separable models $\frk{M},\frk{N}\models T_{\bar{a}}$ and $\e>0$ there is an almost correlation $R$ between $\frk{M}$ and $\frk{N}$ such that $\mathrm{dis}_\Delta(R) < \e$ and for all $(\bar{m},\bar{n}) \in R$, $|d^\frk{M}(\bar{m},\bar{a}^\frk{M})-d^\frk{N}(\bar{n},\bar{a}^\frk{N})|<\e$.
\end{cor}

\begin{proof}
Unpacking definitions gives that $d_{D(\Delta,\bar{a})}$ in $S_n(\bar{a} \bar{b})$ is the same thing as $d_\Delta$ in $S_n(\bar{a}\bar{b})$. So we clearly have that $T_{\bar{a}}$ is $D(\Delta,\bar{a})$-$\omega$-categorical. All we need to do is verify that the alternative statement of $D(\Delta,\bar{a})$-$\omega$-categoricity is equivalent, but this follows from the definition of $D_0(\Delta,\bar{a})$.
\end{proof}

Note that in particular if $\Delta$ is the collection of all formulas, so that $\Delta$\nobreakdash-$\omega$\nobreakdash-\hskip0pt categoricity corresponds to $\omega$-categoricity, then $D(\Delta,\bar{a})$-$\omega$-categoricity is not the same thing as $\omega$-categoricity for $T_{\bar{a}}$. Rather, it says that given any two separable models $\frk{M},\frk{N}\models T_{\bar{a}}$, for any $\e>0$ there is an isomorphism $f:\frk{M}\cong \frk{N}$ such that $d(f(\bar{a}^\frk{M}),\bar{a}^\frk{N}) < \e$. Similar weakenings occur when $\Delta$ is a distortion system with an `obvious' extension to $T_{\bar{a}}$, although in some cases, such as with the Gromov-Hausdorff distance, the obvious extension is equivalent to $D(\Delta,\bar{a})$ (in particular because $\delta_\Delta$ and $d_{\Delta}$ are uniformly equivalent).


Recall the definition of the `elementary Gromov-Hausdorff distance,' given in \cite{Hansona}.
\begin{defn}
For any $\Lcal$-formula $\varphi$, let $\chi_\varphi(\bar{x}) = \inf_{\bar{y}}\varphi(\bar{y})+d(\bar{x},\bar{y})$.  $\chi_\varphi$ has the property that it is $1$-Lipschitz in any $\mathcal{L}$-structure and furthermore that if $\varphi$ is $1$-Lipschitz in every model of $T$, then $T\models \chi_\varphi = \varphi$. Let $\mathrm{eGHK}_0 =\{\chi_\varphi : \varphi \in \Lcal\}$. Note that $\varphi \in \mathrm{eGHK}_0$ for any sentence $\varphi$.

Let $\mathrm{eGHK}= \ol{\mathrm{eGHK}_0}$. $\rho_{\mathrm{eGHK}}(\frk{M},\frk{N})$ is the \emph{elementary Gromov-Hausdorff-Kadets distance between $\frk{M}$ and $\frk{N}$}.
\end{defn}

\begin{cor}
If a countable theory $T$ is $\mathrm{eGHK}$-$\omega$-categorical, then it is $\omega$\nobreakdash-\hskip0pt categorical.
\end{cor}
\begin{proof}
$\delta_{\mathrm{eGHK}} = d$ and for any type $p\in S_n(T)$ (it is important that there are no parameters), $p$ is weakly $d$-atomic if and only if it is $d$-atomic. Every type in $S_n(T)$ being $d$-atomic is equivalent to $S_n(T)$ being metrically compact, so it follows that $T$ is $\omega$-categorical.
\end{proof}

It may seem that there is a contradiction here, given the knowledge that ordinary $\omega$-categoricity isn't always preserved under adding constants. This stems from the fact that given a theory $T$ and the corresponding $\mathrm{eGHK}(T)$, $\mathrm{eGHK}(T)(\bar{a})$ is not the same thing as $\mathrm{eGHK}(T_{\bar{a}})$ and the difference is related to what happens with $\omega$\nobreakdash-categorical theories that fail to be $\omega$-categorical after adding constants. Witnesses for $\rho_{\mathrm{eGHK}(T)(\bar{a})}(\frk{M},\frk{N}) < \e$ for $\frk{M},\frk{N}\models T_{\bar{a}}$ are elementary embeddings $f:\frk{M}\preceq \frk{C}$ and $g:\frk{N}\preceq \frk{C}$ such that $d^\frk{C}_H(f(\frk{M}),g(\frk{N})) < \e$ and $d^\frk{C}(f(\bar{a}),g(\bar{a})) < \e$. On the other hand witnesses for $\rho_{\mathrm{eGHK}(T_{\bar{a}})}(\frk{M}_{\bar{a}},\frk{N}_{\bar{a}}) < \e$ require that $f(\bar{a})=g(\bar{a})$ because in this case we're thinking of the big model $\frk{C}$ as a model of $T_{\bar{a}}$ and we need $f$ and $g$ to be elementary embeddings for $T_{\bar{a}}$, not just $T$.

We were unable to show the analogous result for uncountable cardinalities but we were also unable to construct a counterexample, so a natural question arises.

\begin{quest} \label{quest:eGHK}
Does there exist a countable theory $T$ and an uncountable cardinality $\kappa$ such that $T$ is $\mathrm{eGHK}$-$\kappa$-categorical but not $\kappa$-categorical? 
\end{quest}

This question is obviously trivial in single-sorted discrete theories, but there is a specific case of it with a many-sorted discrete theory that is non-trivial and can be resolved negatively (i.e.\ $\mathrm{eGHK}$-$\kappa$-categoricity implies $\kappa$-categoricity). If we have a many-sorted discrete theory with sorts $\{\mathcal{S}_i\}_{i<\omega}$ and we take the metric on sort $\mathcal{S}_i$ to be $\{0,2^{-i}\}$-valued, then the question is non-trivial. This is equivalent to taking a Morleyized many-sorted language $\Lcal$ and letting $\Lcal_i$ (in the sense of Definition \ref{defn:strat-lang}) be the set of all formulas with free variables in the first $i$ sorts (but, crucially, we're implicitly allowing quantification over arbitrarily high sorts). Assume that such a theory is $\mathrm{eGHK}$-$\kappa$-categorical for some uncountable $\kappa$. By Corollary \ref{cor:discrete-other-direc} in the next section, this implies that it is $\mathrm{GHK}$-$\lambda$-categorical for every uncountable $\lambda$ (since this is a discrete theory). Therefore it cannot have any Vaughtian pairs. Now to show that it is actually uncountably categorical we just need to show that it is $\omega$-stable. What we have is that it is $\mathrm{eGHK}$-$\omega$-stable (as defined in the next section). Unpacking the definition in this case, this implies that it cannot have an infinite binary tree whose parameters all come from a finite collection of sorts (whereas in principle a binary tree in general could have parameters from all of the sorts). However, by stability and in particular the fact that the sorts are stably embedded, all of the formulas in the binary tree are definable with parameters from the sort that the tree is in, so we have that in this case $\mathrm{eGHK}$-$\omega$-stability implies $\omega$-stability. Therefore the theory is actually uncountably categorical. By the results of \cite{CatCon}, this is enough to resolve Question \ref{quest:eGHK} in the context of continuous theories with totally disconnected type spaces (which include $\omega$-stable ultrametric theories) since these are bi-interpretable with many-sorted discrete theories.

\subsubsection{Counterexamples for Separable Categoricity}

The simplest non-trivial example of an approximately $\omega$-categorical theory is easiest to describe in terms of a stratified language, as defined in \cite{Hansona}. We will restate the relevant definition and a relevant result here.

\begin{defn} \label{defn:strat-lang}
A \emph{stratified language} is a language $\Lcal$ together with a designated sequence of sub-languages $\{\Lcal_i\}_{i<\omega}$ whose union is $\Lcal$. (Note that the sub-languages may have fewer sorts than the full language.)

In the context of a stratified language $\Lcal$, two $\Lcal$-structures $\frk{M}$, $\frk{N}$ are said to be \emph{approximately isomorphic}, written $\frk{M} \approxx_\Lcal \frk{N}$, if $\frk{M}\upharpoonright \Lcal_i \cong \frk{N}\upharpoonright \Lcal_i$ for every $i<\omega$. In general let $\rho_\Lcal(\frk{M},\frk{N}) = 2^{-i}$ where $i$ is the largest such that $\frk{M} \upharpoonright \Lcal_i \cong \frk{N} \upharpoonright \Lcal_i$ but $\frk{M} \upharpoonright \Lcal_{i+1} \not\cong \frk{N} \upharpoonright \Lcal_{i+1}$, or $0$ if no such $i$ exists. 

We may drop the subscript $\Lcal$ if the relevant stratified language is clear by context.
\end{defn}

\begin{prop}
Let $T$ be a discrete first-order theory (i.e.\ every predicate is $\{0,1\}$-valued in every model of $T$) and let $\Delta$ be a distortion system for $T$.

\begin{itemize}
    \item[(i)] For every finite set $\mathcal{S}_0\subseteq \mathcal{S}$ there is an $\e > 0$ such that if $\mathrm{dis}_{\Delta}(R) < \e$, then $R$ restricted to the sorts in $\mathcal{S}_0$ is the graph of a bijection. For every predicate symbol $P$ there is an $\e_P > 0$ such that whenever $\mathrm{dis}_{\Delta}(R) < \e_P$, then $R$ is the graph of a bijection that respects $R$.
    \item[(ii)] There is a stratification of $\Lcal$ such that $\rho_\Delta$ and $\rho_\Lcal$ are uniformly equivalent. In particular $\frk{M}\approxx_\Delta\frk{N}$ if and only if $\frk{M} \approxx_\Lcal \frk{N}$. 
\end{itemize}
\end{prop}

Now we can give the counterexample.

\begin{ex}
 Let $\Lcal_0 = \{a_i\}_{i<\omega}\cup \{b_0\}$ and let $\Lcal_{n+1} = \Lcal_n \cup \{b_{n+1}\}$. Consider the $\Lcal$-structure in which all of the constants are assigned to different elements and there are no other elements. The theory of this structure is approximately $\omega$\nobreakdash-categorical. We can describe the distortion system corresponding to this stratified language easily as the one generated by $\Delta_0 = \{d(x,a_i)\}_{i<\omega}\cup\{2^{-j} d(x,b_j)\}_{j<\omega}$, where we're taking the metric to be $\{0,1\}$-valued. Let $\Delta = \ol{\Delta_0}$.
\end{ex}

This serves as a counterexample to three things. It is an example of a $\Delta$\nobreakdash-$\omega$\nobreakdash-\hskip0pt categorical theory that is not $\omega$-categorical. It is an example of a theory in a stratified language that is approximately $\omega$-categorical but which does not have an $\omega$-categorical reduct $T \upharpoonright \Lcal_n$ for any $n$. And finally it is an example of the failure of the converse of Theorem \ref{thm:w-cat-main} part \textit{(iii)}. The unique $1$-type of a non-realized element is weakly $d_\Delta$-atomic, but not $d_\Delta$-atomic, so $S_1(T)$ cannot be metrically compact with regards to $d_\Delta$.

A more subtle hope would be that we could eliminate the parameters from Theorem \ref{thm:w-cat-main}. Unfortunately this is impossible in general, but an example is more complicated.

\begin{ex}
A theory $T$ with distortion system $\Delta$ such that for every $n<\omega$, every types $p\in S_n(T)$ is weakly $d_\Delta$-atomic, but such that $T$ is not $\Delta$\nobreakdash-$\omega$\nobreakdash-categorical.
\end{ex}

\begin{proof}

Let $E$ be a binary relation, let $\{f_i\}_{i<\omega}$ be a sequence of unary functions, and let $\{a_{i,j}\}_{i,j<\omega}$ be an array of constant symbols. 

Let $\Lcal_0=\{E\}\cup \{f_i\}_{i<\omega}\cup\{a_{0,0}\}$. For each $k<\omega$, let $\Lcal_{k+1} = \Lcal_k \cup \{a_{i,j}\}_{i,j\leq k+1}$. Finally let $\Lcal = \bigcup_{k<\omega} \Lcal_k$.

Let $\frk{M}$ be an $\Lcal$ structure with universe $\omega \times (\omega + \omega)$. Set $(i,j)E^\frk{M}(k,\ell)$ if and only if $i=k$. Let $f_k((i,j))=(i,k)$ and $a_{i,j}=(i,\omega + j)$. Finally let $T=\mathrm{Th}(\frk{M})$ and let $\Delta$ be a distortion system equivalent to the stratified language $\Lcal$.

To see that $T$ is not $\Delta$-$\omega$-categorical, let $\frk{N}$ be an elementary extension of $\frk{M}$ that realizes a new $E$-equivalence class, but only the outputs of the functions $f_i$ in that class. $\frk{M}$ and $\frk{N}$ are not $\Delta$-approximately isomorphic since there is nothing to correlate the constants $a_{i,j}$ in $\frk{M}$ with in the new equivalence class in $\frk{N}$.

To see that for every $n<\omega$ and every $p\in S_n(T)$, $p$ is weakly $d_\Delta$-atomic, let $\frk{O}$ be the countable elementary extension of $\frk{M}$ that realizes infinitely many new $E$-equivalence classes and realizes infinitely many elements not equal to any $f_i(c)$ or $a_{i,j}$ in every $E$-equivalence class. This is the unique countable $\omega$-saturated model of this theory. We would like to show that the type of any finite tuple of elements of this model is weakly $d_\Delta$-atomic. It is enough to consider tuples of the following form. Fix $N<\omega$. Consider the tuple whose elements are 

\begin{itemize}
\item The constants $a_{i,j}$ for $i,j\leq N$.
\item $N$ elements $E$ equivalent to $a_{i,0}$ for each $i\leq N$, but not equal to any $f_i(c)$ for any $c$.
\item $N$ elements not equal to any $f_i(c)$ from $N$ distinct $E$ equivalence classes that contain no elements of $\frk{M}$ ($N^2$ elements in total).
\item The image of all of those elements under $f_i$ for $i\leq N$.
\end{itemize}

Let the type of this tuple be $p_N$. Every finitary $\varnothing$-type is the type of some sub-tuple of $p_N$ for some $N$. The restriction of a weakly $d_\Delta$-atomic type to some sub-tuple is still $d_\Delta$-atomic (since projection maps are open and $1$-Lipschitz with regards to $d_\Delta$).  For any $k$, the type $p_N \upharpoonright \Lcal_k$ is realized by a tuple in $\frk{M}$, and therefore in any model of $T$, since $\frk{M}$ is the prime model. Therefore by Proposition \ref{prop:omit-types}, $p_N$ is weakly $d_\Delta$-atomic.
\end{proof}



\subsection{Inseparable Approximate Categoricity} \label{subsec:mor-thm} 

\begin{defn}
For any topological space $X$ and metric $d:X^2 \rightarrow \mathbb{R}$, a \emph{$(d,\e)$\nobreakdash-\hskip0pt perfect tree in $X$} is a family of non-empty closed sets $\{F_\sigma\}_{\sigma \in 2^{<\omega}}$ such that for any $\sigma \in 2^{<\omega}$ and $i<2$, $F_{\sigma \frown i} \subset  \mathrm{int}_X F_\sigma$ and $d_{\inf}(F_{\sigma\frown 0},F_{\sigma \frown 1}) > \e$.
\end{defn}

\begin{defn}
Fix a complete theory $T$ and a distortion system $\Delta$ for $T$.
\begin{itemize}
    \item $T$ is \emph{$\Delta$-$\kappa$-stable} if for any parameter set $A$ of size $\leq \kappa$, $\dc (S_1(A),d_{\Delta,A}) \leq \kappa$.
    \item $T$ is \emph{$\Delta$-totally transcendental} or \emph{$\Delta$-t.t.} if for any parameter set $A$ there does not exist a $(d_{\Delta,A},\e)$-perfect tree in $S_1(A)$.\qedhere
\end{itemize}
\end{defn}

\begin{prop}
Fix a theory $T$ and a distortion system $\Delta$ for $T$.
\begin{itemize}
    \item[(i)] If $T$ is $\Delta$-$\omega$-stable, then it is $\Delta$-t.t.
    \item[(ii)] If $T$ is $\Delta$-t.t., then it is $\Delta$-$\kappa$-stable for any $\kappa \geq |\Lcal|$.
    \item[(iii)] If $T$ is countable and $\Delta$-$\kappa$-stable for some $\kappa = \kappa^\omega$, then it is stable. In particular a $\Delta$-$\omega$-stable countable theory is stable.
\end{itemize}
\end{prop}
\begin{proof}
\textit{(i)}: Assume that for some set of parameters $A$ and some $\e>0$ there is a $(d_{\Delta,A},\e)$-perfect tree $\{F_\sigma\}_{\sigma \in 2^{<\omega}}$ in $S_1(A)$. 

Fix $\sigma \in 2^{<\omega}$. Fix $p\in F_{\sigma \frown 0}$ and $q\in F_{\sigma \frown 1}$. We have by assumption that $d_{\Delta,A}(p,q) > \e$. This implies that there exists a $D(\Delta(A),\bar{c})$-sentence, $\varphi_{\sigma,p,q}$, such that $|\varphi_{\sigma,p,q}(p)-\varphi_{\sigma,p,q}(q)|>\e$. 
 There are open sets $U\ni p$ and $V \ni q$ such that for all $r\in U$ and $s\in V$, $|\varphi_{\sigma,p,q}(r)-\varphi_{\sigma,p,q}(s)| > \e$. Therefore by compactness there is a finite collection $\Sigma_\sigma$ of $D(\Delta(A),\bar{c})$-sentences such that for any $p \in F_{\sigma\frown 0}$ and $q\in F_{\sigma\frown 1}$ there is a $\varphi \in \Sigma_\sigma$ such that $|\varphi(p)-\varphi(q)| > \e$. 

For each $\sigma \in 2^{<\omega}$, let $\psi_\sigma$ be a restricted $A$-formula such that $\cset{\psi_\sigma\leq \frac{2}{3}}\subseteq \mathrm{int}_{S_1(A)}F_\sigma$ and such that $F_{\sigma \frown 0},F_{\sigma \frown 1} \subseteq \cset{\psi_\sigma < \frac{1}{3}}$. 
Now let $A_0$ be the collection of all parameters used in some $\varphi \in \Sigma_\sigma$ or $\psi_\sigma$ for some $\sigma \in 2^{<\omega}$. Note that $A_0$ is a countable set of parameters. For each $\sigma \in 2^{<\omega}$, let $G_\sigma = \cset{\psi_\sigma \leq \frac{2}{3}}$ and note that by construction for any $\sigma \in 2^{<\omega}$, and $i<2$, $G_{\sigma \frown i} \subseteq  \mathrm{int}_{S_1(A_0)} G_\sigma$.

To verify that $\{G_\sigma\}_{\sigma \in 2^{<\omega}}$ is a $(d_{\Delta, A_0},\e)$-perfect tree in $S_1(A_0)$ we now just need to verify that $d_{\Delta,A_0}(p,q)>\e$ for any $p\in G_{\sigma\frown 0}$ and $q\in G_{\sigma \frown 1}$, but this follows easily from the inclusion of the parameters from $\Sigma_\sigma$. For any such $p,q$ there is some sentence $\varphi \in \Sigma_\sigma$ such that $|\varphi(p)-\varphi(q)|>\e$. So we have that $\{G_\sigma\}$ is a $(d_{\Delta,A_0},\e)$\nobreakdash-perfect tree. Therefore $\dc(S_1(A_0),d_{\Delta,A_0}) \geq 2^{\aleph_0}$ and $T$ is not $\Delta$-$\omega$\nobreakdash-stable.

\textit{(ii)}: Suppose that $T$ fails to be $\Delta$-$\kappa$-stable for some $\kappa \geq |\Lcal|$. Let $A$ be a collection of parameters of cardinality $\leq \kappa$ such that $\dc (S_1(A),d_{\Delta,A}) > \kappa$. This implies that there is some $\e >0$ such that $\ent_{>\e} (S_1(A),d_{\Delta,A}) \geq \kappa^+$.

Let $\mathcal{B}$ be a base for the topology of $S_1(A)$ of cardinality $\kappa$ (this exists because $\kappa \geq |\Lcal|$). Define a transfinite sequence of closed subsets.

\begin{itemize}
\item $X_0 = S_1(A)$
\item $X_{\alpha + 1} = X_\alpha \setminus \bigcup \{U\in \mathcal{B}:\ent_{>\e}(U\cap X_\alpha) \leq \kappa\}$
\item $X_\lambda = \bigcap_{\alpha < \lambda} X_\alpha$, for $\lambda$ a limit or $\infty$.
\end{itemize}

For each $U \in \mathcal{B}$, let $\alpha(U)$ be the smallest ordinal such that $\ent_{>\e}(U\cap X_{\alpha(U)}) \leq \kappa$ if it exists, and $\infty$ otherwise. Let $\beta = \sup\{\alpha(U):\alpha(U)<\infty\}$. Since $\kappa^+$ is a regular cardinal, we must have that $\beta < \kappa^+$. In particular this implies that $X_\beta=X_{\beta + 1} = X_\infty$.

Now assume that $X_\beta = \varnothing$. Let $Y$ be a $(>\e)$-separated subset of $S_1(A)$ of cardinality $\kappa^+$ (this must exist because $\kappa^+$ is a regular cardinal). Since $\kappa^+$ is a regular cardinal, there must be some $\alpha$ such that $|Y \cap (X_\alpha\setminus X_{\alpha+1})| = \kappa^+$. Furthermore since $\mathcal{B}$ has cardinality $\kappa$ this implies that there must be a $U \in \mathcal{B}$ such that $|Y \cap (X_\alpha\setminus X_{\alpha+1})\cap U| = \kappa^+$, but this implies that $\ent_{>\e}(U\cap X_\alpha) \geq \kappa^+$, which is a contradiction. Therefore $X_\beta$ is non-empty.

Now $X_\beta$ must have the property that for any $U \in \mathcal{B}$, $\ent_{>\e}(U\cap X_\beta) \geq \kappa^+$, so in particular in any non-empty open subset of $X_\beta$ there exists $p,q$ with $d_{\Delta,A}(p,q) > \e$. 

Let $F_\varnothing = S_1(A)$. For each $\sigma \in 2^{<\omega}$, given $F_\sigma$, a closed set whose interior has non-empty intersection with $X_\beta$, find $p,q \in X_\beta \cap \mathrm{int}_{S_1(A)} F_\sigma$ such that $d_{\Delta,A}(p,q) > \e$. Find $F_{\sigma \frown 0}$, a closed set such that $p \in \mathrm{int}_{S_1(A)} F_{\sigma \frown 0}$, $F_{\sigma \frown 0} \subset  \mathrm{int}_{S_1(A)} F_\sigma$, and such that $F_{\sigma \frown 0} \cap B_{\leq \e}^{d_{\Delta,A}}(q) = \varnothing$. Then find a $F_{\sigma \frown 1}$, a closed set such that $q \in \mathrm{int}_{S_1(A)}F_{\sigma\frown 1}$, $F_{\sigma \frown 1} \subset \mathrm{int}_{S_1(A)} F_\sigma$, and such that $F_{\sigma \frown 1} \cap F_{\sigma \frown 0}^{d_{\Delta,A} \leq \e} = \varnothing$. Then by construction $\{F_\sigma\}_{\sigma \in 2^{<\omega}}$ is a $(d_{\Delta,A},\e)$-perfect tree, so $T$ is not $\Delta$-t.t.

\textit{(iii)}: The cardinality of a complete metric space of density character $\lambda$ is always either $\lambda$ or $\lambda^\omega$. This implies that the cardinality of $S_1(A)$ for any set of parameters $A$ with $|A|\leq \kappa$ must be $\leq \kappa$, since it's $d_{\Delta,A}$-density character is $\leq \kappa = \kappa^\omega$. Therefore $T$ is stable with regards to the discrete metric and in particular stable.
\end{proof}

\begin{prop} \label{prop:w-dense}
For any theory $T$ and distortion system $\Delta$ for $T$, if $T$ is $\Delta$-t.t., then for any set of parameters $A$ and any closed set $F \subseteq S_n(A)$, $d_{\Delta(A)}$\nobreakdash-atomic\nobreakdash-in\nobreakdash-$F$ types are dense in $F$.
\end{prop}

\begin{proof}
This is follows from Corollary 3.8 of \cite{BenYaacov2008}.
\end{proof}

\begin{prop} \label{prop:EM-thin}
For any countable first-order theory $T$, there are models of any density character that only realize separably many types over countable parameter sets.
\end{prop}
\begin{proof}
This follows from Proposition 3.37 of \cite{ben-yaacov_2005}.
\end{proof}

\begin{cor} \label{cor:cat-w-stab}
If a countable theory $T$ with distortion system $\Delta$ is $\Delta$\nobreakdash-$\kappa$\nobreakdash-\hskip0pt categorical for uncountable $\kappa$, then $T$ is $\Delta$-$\omega$-stable, so in particular it is $\Delta$-t.t. and $\Delta$-$\kappa$-stable for every $\kappa \geq \aleph_0$.
\end{cor}
\begin{proof}
Assume that $T$ is not $\Delta$-$\omega$-stable. Then there is some countable parameter set $A$ and an $\e>0$ such that there is an uncountable $(d_{\Delta,A} > \e)$-separated set of types $P$ in $S_1(A)$.  Let $\frk{M}$ be a model of density character $\kappa$ that realizes $A$ and every element of $P$. Let $\frk{N}$ be a model of $T$ of density character $\kappa$ such that for any countable $B \subset \frk{N}$, the set of types in $S_1(B)$ realized in $\frk{N}$ is separable with respect to the $d$-metric. 

Find a $\delta > 0$ small enough that $\delta_\Delta(\mathrm{tp}(ab),\mathrm{tp}(ce)) < \delta$ then $|d(a,b)-d(c,e)| < \frac{\e}{9}$. Find an almost correlation $R$ between $\frk{M}$ and $\frk{N}$ with $\mathrm{dis}_\Delta(R) < \frac{\e}{9} \downarrow \delta$. By replacing each element of $A$ with a Cauchy sequence if necessary we may assume that $A$ is in the domain of $R$ (extending each element of $P$ so that they are still realized in $\frk{M}$, note that they are still $(d_{\Delta,A} > \e)$-separated). Let $B$ be some $\omega$-tuple of elements of $\frk{N}$ correlated with $A$ by $R$.

Now for each $p \in P$, find a type $p\prime$ with $d(p,p^\prime) < \frac{\e}{3}$ such that $p^\prime$ is still realized in $\frk{M}$ but by something in the domain of $R$. $P^\prime$ is now $(d_{\Delta,A} > \frac{\e}{3})$-separated. Now for each $p\prime \in P$, let $q \in S_1(B)$ be a type such that $q$ is realized in $\frk{N}$ by something correlated to a realization of $p^\prime$ in $\frk{M}$. Let $Q$ be the collection of these types.

By unpacking definitions we have that for any countable parameter set $E$, the metric $d_{\Delta,E}$ on $S_1(E)$ is equivalent to the metric $\delta^\omega_{D(\Delta,c)}$ restricted to $S_1(E)$ as a subspace of $S_\omega(T_c)$, where the $\omega$ variables correspond to an enumeration of $E$, $c$ is a fresh constant symbol corresponding to the free variables in the types in $S_1(E)$, and $T_c$ is the (incomplete) theory of $T$ with an extra constant (and no new axioms).

Pick $q_0,q_1 \in Q$, corresponding to $p_0^\prime,p_1^\prime \in P^\prime$. The choice of $\delta$ implies that $\delta^\omega_{D(\Delta,c)}(p_0^\prime,q_0) < \frac{\e}{9}$ and $\delta^\omega_{D(\Delta,c)}(q_1,p_1^\prime) < \frac{\e}{9}$. Since 
\begin{align*}
d_{\Delta,A}(p_0^\prime,p_1^\prime) &\leq \delta^\omega_{D(\Delta,c)}(p_0^\prime,q_0) + \delta^\omega_{D(\Delta,c)}(q_0,q_1) + \delta^\omega_{D(\Delta,c)}(q_1,p_1^\prime), \\
d_{\Delta,A}(p_0^\prime,p_1^\prime) &\leq \delta^\omega_{D(\Delta,c)}(p_0^\prime,q_0) + d(q_0,q_1) + \delta^\omega_{D(\Delta,c)}(q_1,p_1^\prime),
\end{align*}
this implies that $d(q_0,q_1) > \frac{\e}{3} - \frac{2\e}{9} = \frac{\e}{9}$. So $Q$ is $(d>\frac{\e}{9})$-separated which is a contradiction. Therefore $T$ is $\Delta$-$\omega$-stable.
\end{proof}

\begin{defn}
For a countable theory $T$ with distortion system $\Delta$, a model $\frk{M}\models T$ is \emph{$\Delta$-$\kappa$-saturated} if for any $A\subseteq \frk{M}$ with $|A|<\kappa$, $\frk{M}$ realizes a $d_{\Delta,A}$\nobreakdash-dense subset of $S_1(A)$.

If $\dc \frk{M} = \kappa$, we say that $\frk{M}$ is $\Delta$-saturated.
\end{defn}

When $\Delta$ is the collection of all formulas, a structure is $\Delta$-$\kappa$-saturated if and only if it is $\kappa$-saturated.

At this point there is a notable omission. Given $\frk{M},\frk{N}\models T$, both $\Delta$-saturated with the same uncountable density character, it's unclear whether or not we can conclude $\rho_\Delta(\frk{M},\frk{N})=0$. If we could then we would be able to prove the full analog of Morley's theorem for distortion systems.

\begin{prop}
If $T$ is a complete, countable, $\Delta$-$\kappa$-stable theory with non-compact models for $\Delta$, a distortion system for $T$,  then for every uncountable regular $\lambda \leq \kappa$, $T$ has a $\Delta$-$\lambda$-saturated model of density character $\kappa$.
\end{prop}
\begin{proof}
Let $\frk{M}_0$ be any pre-model of density character and cardinality $\kappa$. Form a continuous elementary chain $\{\frk{M}_i\}_{i<\lambda}$ of length $\kappa$ of pre-models of density character and cardinality $\kappa$ such that for each $i$, $\frk{M}_{i+1}$ realizes a $d_{\Delta, \frk{M}_i}$-dense set of types in $S_1(\frk{M}_i)$.

Finally let $\frk{M}$ be the completion of the union. Clearly $\dc \frk{M} = \kappa$. Any subset $A$ of $\frk{M}$ of cardinality $<\lambda$ is in the closure of some $\frk{M}_i$. Then $S_1(\frk{M}_i A) = S_1(\frk{M}_i)$, so we have that a $d_{\Delta,\frk{M}_i A}$-dense subset of $S_1(\frk{M}_iA)$ is realized in $\frk{M}$. This implies that a $d_{\Delta,A}$-dense subset of $S_1(A)$ is realized in $\frk{M}$, and so $\frk{M}$ is $\Delta$-$\lambda$-saturated.
\end{proof}

\begin{cor}
If a countable theory $T$ with distortion system $\Delta$ is $\Delta$\nobreakdash-$\kappa$\nobreakdash-\hskip0pt categorical for some $\kappa \geq \aleph_1$, then every model of $T$ of density character $\kappa$ is $\Delta$-saturated.
\end{cor}
\begin{proof}
Let $\frk{M}$ be a model of $T$ of density character $\kappa\geq \aleph_1$. For any regular uncountable $\lambda \leq \kappa$, let $\frk{N}$ be a $\Delta$-$\lambda$-saturated model of $T$ of density character $\kappa$. Let $A\subset \frk{M}$ be any subset of cardinality $<\lambda$. Pick $p \in S_1(A)$ and $\e > 0$. Find a $\delta > 0$ small enough that $\delta_\Delta(\mathrm{tp}(ab),\mathrm{tp}(ce)) < \delta$ then $|d(a,b)-d(c,e)| < \frac{\e}{3}$.
Find an almost correlation $R$ between $\frk{M}$ and $\frk{N}$ with distortion $<\frac{\e}{3} \downarrow \delta$. Every element of $A$ is a metric limit of points in the domain of $R$. Let $A^\prime$ be a set containing a sequence limiting to each element of $A$ (note that we still have $|A^\prime|< \lambda$) and let $p^\prime$ be some extension of $p$ to $S_1(A^\prime)$. Go to a large enough elementary extension of $(\frk{M},\frk{N},R)$ that $p^\prime$ is realized over $A^\prime$ by some $m$ and such that $R$ is a correlation. Let $B\subseteq \frk{N}$ be correlated to $A^\prime$ by $R$ so that $|B|< \lambda$, and let $e$ be correlated to $m$ in the elementary extension. Let $q$ be $\mathrm{tp}(e/B)$. 

By $\Delta$-$\lambda$-saturation, there is some type $r \in S_1(B)$ such that $\frk{N}$ realizes $r$ with some $e^\prime$ and $d_{\Delta,B}(q,r) < \frac{\e}{5}$. Find $e^{\prime \prime}$ such that $d(e^\prime,e^{\prime \prime})<\frac{\e}{3}$ and such that $e^{\prime \prime}$ is correlated to some $m^\prime \in \frk{M}$ by $R$. 

Just like in the proof of Corollary \ref{cor:cat-w-stab}, we have that 
\begin{align*}
d_{\Delta,A}(p^\prime,\mathrm{tp}(m^\prime/A^\prime)) &\leq \delta^\omega_{D(\Delta(A),c)}(p^\prime,q)\\
&+d(e^\prime,e^{\prime \prime})
\\&+\delta^\omega_{D(\Delta(A),c)}(\mathrm{tp}(e^{\prime \prime}/B),\mathrm{tp}(m^\prime/A^\prime)),
\end{align*}
 where we're thinking of $A$ and $B$ as $\omega$-tuples when computing $\delta^\omega_{D(\Delta(A),c)}$. So we have that $d_{\Delta,A}(p^\prime,\mathrm{tp}(m^\prime/A^\prime)) < \frac{\e}{3} + \frac{\e}{3} + \frac{\e}{3} = \e$. 

Since we can do this for any $A$ with $|A| < \lambda$, any $p\in S_1(A)$, and any $\e>0$, we have that $\frk{M}$ is $\Delta$-$\lambda$-saturated. Since we can do this for any regular $\lambda \leq \kappa$, we have that $\frk{M}$ is $\Delta$-saturated. 
\end{proof}

\begin{defn}
We say that $T$ is \emph{weakly $\Delta$-$\kappa$-categorical} if every model of $T$ of density character $\kappa$ is $\Delta$-saturated.
\end{defn}

It's immediate from Proposition \ref{prop:EM-thin} that a weakly $\Delta$-$\kappa$-categorical theory for some uncountable $\kappa$ is $\Delta$-$\omega$-stable.

\begin{defn}
Let $T$ be a complete theory and $\Delta$ a distortion system for $T$.
\begin{itemize}
    \item For any parameter set $A$, a \emph{$\Delta$-construction over $A$} is a sequence $\{b_i\}_{i<\lambda}$ such that for every $i<\lambda$, $\mathrm{tp}(b_i/Ab_{<i})$ is $d_{\Delta,Ab_{<i}}$-atomic.
    \item A model $\frk{M}\models T$ is \emph{$\Delta$-constructible over $A$} if $A\subseteq \frk{M}$ and there is an enumeration of a dense subset of $\frk{M}$ which is a $\Delta$-construction over $A$. \qedhere
\end{itemize}
\end{defn}

\begin{prop}
For any complete theory $T$, $\Delta$, a distortion system for $T$, and parameter set $A$, if $\{b_i\}_{i<\lambda}$ is a $\Delta$-construction over $A$, then for any finite tuple $\bar{b} \in \{b_i\}$, $\mathrm{tp}(\bar{b}/A)$ is $d_{\Delta,A}$-atomic.
\end{prop}
\begin{proof}
We have already done most of the work in Proposition \ref{prop:atom-iff} and Lemma \ref{lem:atom-iff-w}. We just need to apply it inductively. Assume that we're shown that for any $i_0 < i_1 <  \dots < i_k < \alpha$ that $\mathrm{tp}(b_{i_0}b_{i_1}\dots b_{i_k} / A)$ is $d_{\Delta,A}$-atomic.  Let $j_0 < j_1 < \dots < j_\ell \leq \alpha$ be any tuple of indices and consider $b_{j_0}b_{j_1}\dots b_{j_k}$. By Proposition  \ref{prop:atom-iff} we know there exists a countable set of elements of the form $b_i$ with $i<\alpha$ such that $\mathrm{tp}(b_{j_0}b_{j_1}\dots b_{j_k}/AB)$ is $d_{\Delta,A}$-atomic. By Lemma \ref{lem:atom-iff-w} this implies that $\mathrm{tp}(b_{j_0}b_{j_1}\dots b_{j_k}/A)$ is $d_{\Delta,A}$-atomic, as required.
\end{proof}

\begin{prop}
Fix $T$, a countable complete $\Delta$-$\omega$-stable theory with non-compact models, for $\Delta$, a distortion system for $T$. If $\frk{M}\models T$ with $\dc \frk{M} = \kappa \geq \aleph_1$, then for any regular $\lambda \leq \kappa$, $\frk{M}$ has arbitrarily large elementary extensions $\frk{N}$ such that for any $A\subset \frk{M}$ with $|A| < \lambda$ the set of types realized by $\frk{N}$ in $S_1(A)$ is contained in the $d_{\Delta,A}$-closure of the set of types realized by $\frk{M}$ in $S_1(A)$.
\end{prop}
\begin{proof}
Fix regular $\lambda \leq \kappa$. Find an $\e >0$ such that $\ent_{>\e} (\frk{M},d) \geq \lambda$, which must exist because $\lambda$ is regular. Let $Q$ be a maximal $({>}\e)$-separated subset of $\frk{M}$ of cardinality $\geq \lambda$. Let $X \subset S_1(\frk{M})$ be the set of all types $p$ such that for every open neighborhood $U \ni p$, $|U\cap Q| \geq \lambda$. Note that $X$ is a closed set and by compactness is non-empty. Also note that since $Q$ was chosen to be maximal, $\frk{M} \cap X = \varnothing$, since for any realized type $p$ there is some $a\in \frk{M}$ such that $d(p,a)<\e$, which is an open neighborhood of $p$ whose intersection with $Q$ has cardinality $1$, which is in particular less than $\lambda$.

Now since $T$ is $\Delta$-$\omega$-stable, $d_{\Delta,\frk{M}}$-atomic-in-$X$ types are dense in $X$. Let $p$ be some $d_{\Delta,\frk{M}}$-atomic-in-$X$ type. Let $a$ be a realization of that type. By Proposition \ref{prop:w-dense}, there exists a $\Delta$-constructible model $\frk{N}$ over $\frk{M}a$ (specifically, $d_{\Delta,A}$-atomic types are dense in every $S_1(A)$). Note that $\frk{N}$ is a proper elementary extension of $\frk{M}$.

Let $b$ be some element of $\frk{N}$ and let $E\subset \frk{M}$ be some set of parameters of cardinality $\mu < \lambda$. Since $\frk{N}$ is $\Delta$-constructible over $\frk{M}a$, $\mathrm{tp}(b/\frk{M}a)$ is $d_{\Delta,\frk{M}a}$-atomic. Let $\varphi(x,\bar{m},a)$ be a formula (with $\bar{m}$ possibly an $\omega$-tuple) such that $\frk{N}\models \varphi(b,\bar{m},a) \leq 0$ and for every type $q \in S_1(\frk{M}a)$, $d_{\Delta,\frk{M}a}(q,\mathrm{tp}(b/\frk{M}a)) \leq \varphi(q,\bar{m},a)$. In particular what this means is that if we think of $c$ as a fresh constant symbol, then for any $D(\Delta(\frk{M}a),c)$-sentence $\psi(c,\bar{m},a)$, $|\psi^\frk{N}(b,\bar{m},a)-\psi(q,\bar{m},a)|\leq \varphi(q,\bar{m},a)$. We may assume that $\bar{m} \in E$. Let $\{\psi_i^0(c,\bar{e}_i)\}_{i<\mu}$ be a collection of $D(\Delta(\frk{E}),c)$-sentences which are dense in the uniform norm (treating $c$ as a variable). Then for each $i<\mu$, let $\psi_i(c,\bar{e}_i) =|\psi_i^0(c,\bar{e}_i) - \psi_i^{0,\frk{N}}(b,\bar{e}_i)|$ and note that each $\psi_i(c,\bar{e}_i)$ is also a $D(\Delta(E),c)$-sentence. Now we have that for each $i<\mu$ and for any $q \in S_1(\frk{M}a)$, $\frk{N} \models \sup_x \psi_i(x,\bar{e}_i) \dotdiv \varphi(x,\bar{m},a)$ and of course $\frk{N} \models \inf_x \varphi(x,\bar{m},a)$, so these two statements are parts of $\mathrm{tp}(a/\frk{M})=p$.  

Let $\chi$ be a $\frk{M}$-formula such that $\chi(p) = 0$ and for every $q\in X$, $d_{\Delta,\frk{M}}(p,q) \leq \chi(q)$, so in particular for any $D(\Delta(\frk{M}),c)$-sentence $\theta(c,\bar{m})$, we have $|\theta(p,\bar{m})-\theta(q,\bar{m})|\leq \chi(q)$. Again by extending $E$ we may assume that $\chi$ is an $E$-formula.

Now pick $\delta>0$, and find a $\gamma > 0$ such that if $d_{\Delta,\frk{M}}(\mathrm{tp}(a^\prime b^\prime/\frk{M}),\mathrm{tp}(a^{\prime \prime}b^{\prime \prime}/\frk{M})) \leq \gamma$ then $|\varphi(b^\prime,\bar{m},a^\prime)-\varphi(b^{\prime\prime},\bar{m},a^{\prime\prime})| < \frac{\delta}{3}$ and note that this implies that if $d_{\Delta,\frk{M}}(\mathrm{tp}(a^\prime/\frk{M}),\allowbreak\mathrm{tp}(a^{\prime \prime}/\frk{M})) \allowbreak \leq \gamma$ then for any $b^\prime$, $|\varphi(b^\prime,\bar{m},a^\prime)-\varphi(b^\prime,\bar{m},a^{\prime\prime})| < \frac{\delta}{3}$ and also the same thing for formulas of the form $\sup_x \psi_i (x,\bar{e}_i) \dotdiv \varphi(x,\bar{m},y)$.  Consider the sets
$$H_{-1} = Q \cap \cset{\chi < \gamma} \setminus \cset*{\inf_x \varphi(x,\bar{m},y) < \frac{\delta}{2}}$$ and $$H_i = Q \cap \cset{\chi < \gamma} \setminus \cset*{\sup_x \psi_i (x,\bar{e}_i) \dotdiv \varphi(x,\bar{m},y) < \frac{\delta}{2}}$$ for $i<\mu$ (as subsets of $S_1(\frk{M})$). Note that since $\cset{\chi < \gamma}$ is an open neighborhood of $p$, it must be the case that $|Q\cap \cset{\chi < \gamma}|\geq \lambda$. Now assume that $|H_i|$ is $\geq \lambda$. This implies that $X\cap \mathrm{cl}_{S_1(\frk{M})} H_i \neq \varnothing$, so in particular there is some $q\in X$ such that $\chi(q) \leq \gamma$ and yet either $\inf_x \varphi(x,\bar{m},q) \geq \frac{\delta}{2}$ or $\sup_x \psi_i (x,\bar{e}_i) \dotdiv \varphi(x,\bar{m},q) \geq \frac{\delta}{2}$, which contradicts the choice of $\gamma$, so it must be the case that $|H_i|< \lambda$ for every $i$ with $-1 \leq i < \mu$. Therefore since $\lambda$ is a regular cardinal we must have that $$\left| Q \cap \cset{\chi < \gamma} \cap \cset*{\inf_x \varphi(x,\bar{m},y) < \frac{\delta}{2}} \cap \bigcap_{i<\mu} \cset*{\sup_x \psi_i (x,\bar{e}_i) \dotdiv \varphi(x,\bar{m},y) < \frac{\delta}{2}} \right| \geq \lambda,$$
so there exists some $a^\prime \in Q$ such that $\frk{M} \models \inf_x \varphi(x,\bar{m},a^\prime) < \frac{\delta}{2}$ and $\sup_x \psi_i (x,\bar{e}_i) \allowbreak \dotdiv \varphi(x,\bar{m},a^\prime) < \frac{\delta}{2}$ for each $i<\mu$. Let $b^\prime$ be an element of $\frk{M}$ such that $\frk{M} \models \varphi(b^\prime,\bar{m},a^\prime) <\frac{\delta}{2}$. Then we have that for each $i<\mu$, $\frk{M} \models \psi_i (b^\prime,\bar{e}_i) \dotdiv \varphi(b^\prime,\bar{m},a^\prime) < \frac{\delta}{2}$, so together this implies that $\frk{M} \models \psi_i(b^\prime,\bar{e}_i) < \delta$ for every $i<\mu$. By the choice of the $\psi_i$'s, this implies that $d_{\Delta,E}(\mathrm{tp(b^\prime/E)},\mathrm{tp}(b/E)) \leq \delta$. Since we can do this for any $\delta > 0$, we have that the set of types in $S_1(E)$ realized in $\frk{N}$ is in the $d_{\Delta,E}$-metric closure of the set of types in $S_1(E)$ realized in $\frk{M}$. 

Now we are free to iterate this process to form arbitrarily large elementary extensions $\frk{N}^\prime \succ\frk{M}$ such that for any set of parameters $E \subset \frk{M}$ with $|E| < \lambda$, if $\frk{N}^\prime$ realizes some type in $S_1(E)$, then it is in the $d_{\Delta,E}$-metric closure of the types in $S_1(E)$ realized in $\frk{M}$.
\end{proof}

\begin{cor} \label{cor:downwards-Morley}
Fix $T$ a countable complete theory with non-compact models and $\Delta$, a distortion system for $T$. For any $\kappa \geq \aleph_1$, if $T$ is weakly $\Delta$\nobreakdash-$\kappa$\nobreakdash-categorical, then it is weakly $\Delta$-$\lambda$-categorical for every $\lambda \geq \aleph_1$ with $\lambda \leq \kappa$.
\end{cor}
\begin{proof}
Assume that $T$ fails to be weakly $\Delta$-$\kappa$-categorical. Let $\frk{M}$ be a model of density character $\kappa$ that fails to be $\Delta$-$\kappa$-saturated. Let $A\subset \frk{M}$ be a set of parameters with $|A| < \kappa$ such that for some type $p \in S_1(A)$ and some $\e>0$, every type $q \in S_1(A)$ realized in $\frk{M}$ has $d(p,q)\geq \e$. Then by the previous proposition $\frk{M}$ has arbitrarily large elementary extensions with the same property, so in particular for any $\lambda > \kappa$, $T$ has a model of density character $\lambda$ that fails to be $\Delta$-saturated, so $T$ is not weakly $\Delta$-$\lambda$-categorical.
\end{proof}

\begin{thm}
Fix $T$ a countable complete theory with non-compact models and $\Delta$, a distortion system for $T$. For any $\kappa \geq \aleph_1$, if $T$ is weakly $\Delta$\nobreakdash-$\kappa$\nobreakdash-categorical, then it is weakly $\Delta$-$\lambda$-categorical for every $\lambda \geq \aleph_1$.
\end{thm}
\begin{proof}
By the previous Corollary \ref{cor:downwards-Morley}, we know that the collection of uncountable cardinalities for which a theory $T$ is weakly $\Delta$-$\kappa$-categorical is always an initial segment of the uncountable cardinals. Assume that it is not all of them. Find $\kappa$ large enough that $T$ is not weakly $\Delta$-$\kappa$-categorical and such that for any $\lambda < \kappa$, $\lambda^\omega < \kappa$ and such that $\mathrm{cf}(\kappa)\geq \omega_1$ (such a cardinal can always be found). 

Let $\frk{M}$ be a model of $T$ of cardinality $\kappa$ which is not $\Delta$-saturated. Let $A\subset \frk{M}$ be a set of parameters with $|A| < \kappa$ such that for some type $p \in S_1(A)$ and some $\e>0$, every type $q \in S_1(A)$ realized in $\frk{M}$ has $d_{\Delta,A}(p,q)\geq \e$ (i.e.\ $A$ witnesses that $\frk{M}$ is not $\Delta$-saturated). Since for any superset $A^\prime \supseteq A$, the natural restriction map $(S_1(A^\prime),d_{\Delta,A^\prime})\rightarrow (S_1(A),d_{\Delta,A})$ is $1$-Lipschitz, we can freely pass to a larger set of parameters and preserve this condition. So we may assume that $|A|=|A|^\omega$ by passing to a larger parameter set if necessary (we can do this since we have ensured that $\lambda^\omega < \kappa$ for any $\lambda < \kappa$). Let $\lambda = |A|$. Since $T$ is $\Delta$-$\omega$-stable, it is stable and so in particular $\lambda$-stable.

Find sufficiently small $\e>0$ such that we can find $\{b_i\}_{i<\lambda^+}$, a $({>}\e)$-separated sequence of elements of $\frk{M}$. By Proposition 4.17 of \cite{ben-yaacov_2005} there exists (in the monster model) an array $\{c_i^j\}_{i<\lambda^+,j<\omega}$ such that $d(c_i^j,c_i^{j+1}) < 2^{-j}$, such that for each $j<\omega$ there is a sub-sequence $I\subseteq \lambda^+$ such that $\{c_i^j\}_{i<\lambda^+}\equiv_A \{b_i\}_{i\in I}$, and such that the sequence of limits $\{c_i^\omega\}_{i<\lambda^+}$ is an $A$-indiscernible sequence. Let $C = \{c_i^\omega\}_{i<\omega}$.

Let $\Sigma$ be a countable dense subset of the collection of finitary $D(\Delta,x)$ formulas (where $x$ is being treated as the fresh constant symbol). Formulas in $\Sigma$ are of the form $\varphi(x,\bar{y})$. The important thing is that $\Sigma$ has the property that for any set of parameters $E$, $d_{\Delta,E}(\mathrm{tp}(f_0/E),\mathrm{tp}(f_1/E)) = \sup_{\varphi \in \Sigma,\bar{e} \in E} |\varphi(f_0,\bar{e})-\varphi(f_1,\bar{e})|$. We may assume that $\Sigma$ is closed under $\varphi \mapsto (\varphi \dotdiv r)$ for each rational $r$.

$(\ast)$ Note that for each restricted $AC$-formula $\varphi(x,\bar{a},\bar{c})$ such that $\models \inf_x \varphi(x,\allowbreak\bar{a},\allowbreak\bar{c}) \leq 0$,  there is an $A$-formula $\psi_{\varphi(-,\bar{a})}(x,\bar{a}_{\varphi(-,\bar{a},\bar{c})})$ with $\psi \in \Sigma$  such that $p(x)\vdash \allowbreak \psi_{\varphi(-,\bar{a},\bar{c})}(x,\allowbreak\bar{a}_{\varphi(-,\bar{a})}) \leq 0$ and such that $$\left\{\varphi(x,\bar{a},\bar{c})<\frac{1}{2},\psi_{\varphi(-,\bar{a},\bar{c})}(x,\bar{a}_{\varphi(-,\bar{a},\bar{c})}) > \frac{\e}{2} \right\}$$ is consistent. This holds because every type realized in $\frk{M}$ has $d_{\Delta,A}$-distance $\geq \e$ from $p$ and because we can approximate $C$ with something of the form $\{b_i\}_{i\in I}$ for some $I\subseteq \lambda^+$, which is a set of elements in $\frk{M}$.

Now let $A_0 = \varnothing$, and for each $n<\omega$, given $A_n$, let $A_{n+1}$ be the collection of all $a$'s occurring in some tuple of the form $\bar{a}_{\varphi(-,\bar{a},\bar{c})}$ where $\varphi(x,\bar{a},\bar{c})$ is a restricted $A_nC$-formula. Finally let $A_\omega = \bigcup_{n<\omega}A_n$. Clearly $A_\omega$ is a countable set.

Now by construction we have that condition $(\ast)$ holds with $A_\omega$ replacing $A$. Also note that $C$ is still an indiscernible sequence over $A_\omega$ (because $A_\omega \subseteq A$). Let $C^\prime$ be an extension of $C$ to an indiscernible sequence of length $\omega_1$. Now let $\frk{N}$ be a $\Delta$\nobreakdash-constructible model over $A_\omega C^\prime$. Note that $\dc \frk{N} = \aleph_1$. If $q$ is the restriction of $p$ to $A_\omega$, we would like to argue that for every type $r\in S_1(A_\omega)$ realized in $\frk{N}$, $d_{\Delta,A_\omega}(q,r) \geq \frac{\e}{2}$. Assume that this is false and that there is some $e \in \frk{N}$ such that $d_{\Delta,A}(q,\mathrm{tp}(e/A_\omega)) \leq \delta < \frac{\e}{2}$. By construction, $\mathrm{tp}(e/A_\omega C^\prime)$ is $d_{\Delta,A_\omega C^\prime}$-atomic. This implies that there is a restricted formula $\chi(x,\bar{a},\bar{c})$ such that $\models \chi(e,\bar{a},\bar{c}) \leq 0$ and for any $f$ if $\models \chi(f,\bar{a},\bar{c}) < \frac{1}{2}$, then $d_{\Delta,A}(\mathrm{tp}(e/A_\omega C^\prime),\mathrm{tp}(f/A_\omega C^\prime)) \leq \frac{\e}{2} - \delta$. In particular this implies that $d_{\Delta,A}(\mathrm{tp}(f/A_\omega C),q)  \leq \frac{\e}{2}$. Therefore for every $A$\nobreakdash-formula $\eta(x,\bar{a})$ with $\eta \in \Sigma$ such that $q(x) \vdash \eta(x,\bar{a}) \leq 0$, we have that $\models \eta(f,\bar{a}) \leq \frac{\e}{2}$. The fact 
\[\forall x \left(\chi(x,\bar{a},\bar{c}) < \frac{1}{2} \rightarrow \eta(x,\bar{a}) \leq \frac{\e}{2}\right)\]
is an elementary property of $\bar{c}$. In particular, it can be expressed as a closed condition. By indiscernibility this is true of every tuple $\bar{c}^\prime \in C$ with the same order type as $\bar{c}$, but this is inconsistent with the modified condition $(\ast)$ (replacing $A$ with $A_\omega$) above, since this implies that $\{\chi(x,\bar{a}, \bar{c}^\prime) < \frac{1}{2}, \psi_{\chi(-,\bar{a},\bar{c}^\prime)}(x,\bar{a}_{\chi(-,\bar{a},\bar{c}^\prime)}) > \frac{\e}{2}\}$ is inconsistent (by setting $\eta(x,\bar{a})$ to $\psi_{\chi(-,\bar{a},\bar{c}^\prime)}(x,\bar{a}_{\chi(-,\bar{a},\bar{c}^\prime)})$). Therefore no such $e$ can exist and for any type $r \in S_1(A_\omega)$ realized in $\frk{N}$, $d_{\Delta,A_\omega}(q,r) \geq \frac{\e}{2}$. Therefore $\frk{N}$ is not $\Delta$-saturated and $T$ is not weakly $\Delta$-$\aleph_1$-categorical.

But this contradicts our assumption, so there is no largest uncountable $\kappa$ such that $T$ is weakly $\Delta$-$\kappa$-categorical and in fact $T$ is weakly $\Delta$-$\kappa$-categorical for all uncountable $\kappa$.
\end{proof}

\begin{quest}
If $T$ is weakly $\Delta$-$\kappa$-categorical for some $\kappa \geq \aleph_1$, does it follow that $T$ is $\Delta$-$\kappa$-categorical?
\end{quest}

In the particular case of a discrete theory with a stratified language we can resolve this positively.

\begin{prop}
Suppose $\Delta$ is a distortion system for some complete discrete theory $T$ equivalent to some stratified language $\Lcal$ and suppose that $\frk{M},\frk{N}\models T$ are $\Delta$\nobreakdash-saturated models of the same cardinality, then they are $\Delta$-approximately isomorphic.
\end{prop}
\begin{proof}
For each $n<\omega$, $\frk{M}$ and $\frk{N}$ are saturated as $\Lcal_n$-structures, therefore they are isomorphic as $\Lcal_n$-structures. Thus they are $\Delta$-approximately isomorphic.
\end{proof}

\begin{cor} \label{cor:discrete-other-direc}
Suppose $\Delta$ is a distortion system for some complete discrete theory $T$ equivalent to some stratified language $\Lcal$. If $T$ is $\Delta$-$\kappa$-categorical for some uncountable $\kappa$ then for every uncountable $\lambda$, $T$ is $\Delta$-$\lambda$-categorical.
\end{cor}



\subsection{Some Examples and the Relationship between Different Notions of Categoricity}

\begin{table}
\centering
\begin{tabular}{l|l|l|l|l|}
\cline{2-5}
                                              				& $\omega_1$                     & $\mathrm{Lip}$-$\omega_1$      		& $\mathrm{GH}$-$\omega_1$      & None                          \\ \hline
\multicolumn{1}{|l|}{$\omega$}                		& Trivial                            	  & Unknown                        					    & Unknown                              & Trivial                            \\ \hline
\multicolumn{1}{|l|}{$\mathrm{Lip}$-$\omega$} 	 & Ex.\ \ref{ex:Lip-Iso}  		& Ex.\ \ref{ex:Lip-Lip}					  & Unknown                              & Ex.\ \ref{ex:Lip-NA}  \\ \hline
\multicolumn{1}{|l|}{$\mathrm{GH}$-$\omega$}	 & Ex.\ \ref{ex:GH-Iso}		 & Ex.\ \ref{ex:GH-Lip} 					& Ex.\ \ref{ex:GH-GH} 		& Ex.\ \ref{ex:GH-NA} \\ \hline
\multicolumn{1}{|l|}{None}                 			   & Trivial                          	    & Ex.\ \ref{ex:NA-Lip} 				 & Ex.\ \ref{ex:NA-GH}		& Trivial                             \\ \hline
\end{tabular}
\caption[Combinations of approximate categoricity conditions.]{Known combinations of separable and inseparable ordinary, Lipschitz, and Gromov-Hausdorff categoricity for metric space theories.}
\label{table:table}
\end{table}

This section is a case study of the relationship between ordinary categoricity and Lipschitz and Gromov-Hausdorff approximate categoricity in the theories of metric spaces. The results are summarized in the Table \ref{table:table}, where `$\kappa \in \{\omega,\omega_1\}$' means $\kappa$-categorical, `$\mathrm{Lip}$-$\kappa$' means $\mathrm{Lip}$-$\kappa$-categorical and not $\kappa$-categorical, `$\mathrm{GH}$\nobreakdash-$\kappa$' means $\mathrm{GH}$-$\kappa$-categorical and not $\mathrm{Lip}$-$\kappa$-categorical, and `none' means not $\mathrm{GH}$\nobreakdash-$\kappa$\nobreakdash-categorical.  It's not hard to prove that $\delta_{\mathrm{Lip}}$ uniformly dominates $\delta_{\mathrm{GH}}$ and that therefore $\mathrm{Lip}$-$\kappa$-categoricity implies $\mathrm{GH}$-$\kappa$-categoricity. The boxes labeled `Trivial' are trivial in the sense that it is very easy to encode discrete structures in finite languages as metric spaces \cite{MetSpaUniv}, and to verify that such structures fall in the corresponding groups here. 
Of course we haven't proven that $\Delta$-$\omega_1$-categoricity is equivalent to $\Delta$-$\kappa$-categoricity for all uncountable $\kappa$, but needless to say if we had a counterexample we would have mentioned it by now.

`Unknown' indicate combinations that are not currently known to be possible. There seems to be a general phenomenon where the combination of ordinary $\omega$\nobreakdash-categoricity and strictly approximate $\omega_1$-categoricity is impossible. In the case of theories with a $\{0,1\}$-valued metric (although allowing $[0,1]$-valued predicates), $\omega$-categoricity implies that the $\varnothing$-type spaces are all finite, so any such theory is interdefinable with a purely discrete theory. That said a distortion system for such a theory could still be non-trivial if the theory does not admit quantifier elimination to a finite language. For any $\omega$-categorical discrete theory that admits quantifier elimination to a finite language, all distortion systems are uniformly equivalent to isomorphism, so here we clearly have that $\Delta$-$\omega_1$-categoricitiy implies $\omega_1$-categoricity. This suggests a pair of purely discrete questions.

\begin{quest} \label{quest:disc-strat}
\textit{(i)} Does there exist a discrete theory $T$ in a stratified language $\Lcal$ such that $T$ is $\omega$-categorical but only approximately $\omega_1$-categorical?

\textit{(ii)} Does there exist a discrete theory $T$ in a stratified language $\Lcal=\bigcup_{i<\omega}\Lcal_i$ such that $T$ is $\omega$-categorical, $T\upharpoonright \Lcal_i$ is $\omega_1$-categorical for every $i<\omega$, but $T$ is not $\omega_1$-categorical?
\end{quest}

Note that a positive answer for part \textit{(ii)} would imply a positive answer for part \textit{(i)}. An example of \textit{(ii)} would have to be rather strange. It is easy to show that such a $T$ cannot have any Vaughtian pairs, so $T$ must fail to be $\omega$-stable. Since it is $\omega$-categorical this would imply that it is strictly stable. There is really only one known strictly stable $\omega$-categorical theory, constructed by Hrushovski, but it has a finite language, whereas an example of what we need would necessarily have an infinite language.

Now we turn to the examples in the chart. The following two examples are very similar. The idea is to encode a sequence of constants in a structureless set in increasingly `harder to detect' ways.

\begin{ex} \label{ex:Lip-Iso}
A metric space theory that is strictly $\mathrm{Lip}$-$\omega$-categorical and $\omega_1$\nobreakdash-\hskip0pt categorical.
\end{ex}
\begin{desc}
Let $\frk{M}$ be a metric space whose universe is $\omega \times \{0,1\}$ with $d((i,j),\allowbreak(k,\ell)) = 1$ if $i \neq k$ and $d((i,0),(i,1)) = \frac{1}{2} + 2^{-i-2}$. Let $T = \mathrm{Th}(\frk{M})$.
\end{desc}

\begin{ex} \label{ex:GH-Iso}
A metric space theory that is strictly $\mathrm{GH}$-$\omega$-categorical and $\omega_1$\nobreakdash-\hskip0pt categorical.
\end{ex}
\begin{desc}
Let $\frk{M}$ be a metric space whose universe is $\omega \times \{0,1\}$ with $d((i,j),\allowbreak(k,\ell)) = 1$ if $i \neq k$ and $d((i,0),(i,1)) = 2^{-i-1}$. Let $T = \mathrm{Th}(\frk{M})$.
\end{desc}
The next two examples are also similar to each other.
\begin{ex} \label{ex:GH-NA}
A metric space theory that is strictly $\mathrm{GH}$-$\omega$-categorical and not $\mathrm{GH}$\nobreakdash-$\omega_1$\nobreakdash-\hskip0pt categorical.
\end{ex}
\begin{desc}
Let $\frk{M}$ be a metric space whose universe is $[0,\frac{1}{2}]\times \{0,1\}$ with $d((x,i),\allowbreak(y,j))= 1$ if $x\neq y$ and $d((x,0),(x,1)) = x$.  Let $T = \mathrm{Th}(\frk{M})$.
\end{desc}

\begin{ex} \label{ex:Lip-NA}
A metric space theory that is strictly $\mathrm{Lip}$-$\omega$-categorical and not $\mathrm{GH}$\nobreakdash-$\omega_1$\nobreakdash-\hskip0pt categorical.
\end{ex}
\begin{desc}
Let $\frk{M}$ be a metric space whose universe is $[\frac{1}{4},\frac{1}{2}]\times \{0,1\}$ with $d((x,i),\allowbreak(y,j))= 1$ if $x\neq y$ and $d((x,0),(x,1)) = x$.  Let $T = \mathrm{Th}(\frk{M})$.
\end{desc}

The following Examples \ref{ex:NA-Lip} and \ref{ex:NA-GH} are the prototypes for the subsequent Examples \ref{ex:Lip-Lip}, \ref{ex:GH-Lip}, and \ref{ex:GH-GH}. The idea is to encode as a metric space a structure that is $\mathbb{Z}$-chains with maps into $[-1,1]$ of the form $\cos(n+\theta)$. Any two such chains are `approximately isomorphic' regardless of their values of $\theta$, since $1$ radian is an irrational rotation.

\begin{ex} \label{ex:NA-Lip}
A metric space theory that is not $\mathrm{GH}$-$\omega$-categorical and is strictly $\mathrm{Lip}$\nobreakdash-$\omega_1$\nobreakdash-\hskip0pt categorical.
\end{ex}
\begin{desc}
Let $\frk{M}$ be a metric space whose universe is $\mathbb{Z} \times \{0,1\}$ with $d((n,i),\allowbreak(m,j))=1$ if $|n-m| > 1$, $d((n,i),(n+1,j)) = \frac{1}{2}$, and $d((n,0),(n,1)) = \frac{1}{4} + \frac{1}{8}\cos(n)$. Let $T = \mathrm{Th}(\frk{M})$.
\end{desc}

\begin{ex} \label{ex:NA-GH}
A metric space theory that is not $\mathrm{GH}$-$\omega$-categorical and is strictly $\mathrm{GH}$-$\omega_1$\nobreakdash-\hskip0pt categorical.
\end{ex}
\begin{desc}
Let $\frk{M}$ be a metric space whose universe is $\mathbb{Z} \times \{0,1\}$ with $d((n,i),\allowbreak(m,j))=1$ if $|n-m| > 1$, $d((n,i),(n+1,j)) = \frac{1}{2}$, and $d((n,0),(n,1)) = \frac{1}{8} + \frac{1}{8}\cos(n)$. Let $T = \mathrm{Th}(\frk{M})$.
\end{desc}

Now we modify the previous examples with `increasingly hard to detect' constants, analogous to Examples \ref{ex:Lip-Iso} and \ref{ex:GH-Iso}. This forces the separable model to have infinitely many $\mathbb{Z}$-chains.

\begin{ex} \label{ex:Lip-Lip}
A metric space theory that is strictly $\mathrm{Lip}$-$\omega$-categorical and strictly $\mathrm{Lip}$-$\omega_1$-categorical.
\end{ex}
\begin{desc}
Let $\frk{M}$ be a metric space whose universe is $\mathbb{N} \times \mathbb{Z} \times {0,1}$ with 
\begin{itemize}
\item $d((a,n,i),(b,m,j))=1$ if $a\neq b$ or if $a=b$ and $|n-m| > 1$,
\item $d((a,n,i),(a,n+1,j)) = \frac{1}{2}$,
\item $d((a,n,0),(a,n,1)) =  \frac{1}{4} + \frac{1}{8}\cos(n)$ if $n\neq 0$, and
\item $d((a,0,0),(a,0,1)) = \frac{1}{4} + \frac{1}{8} + 2^{-a-4}$.
\end{itemize}
Let $T = \mathrm{Th}(\frk{M})$.
\end{desc}

\begin{ex} \label{ex:GH-Lip}
A metric space theory that is strictly $\mathrm{GH}$-$\omega$-categorical and strictly $\mathrm{Lip}$-$\omega_1$-categorical.
\end{ex}
\begin{desc}
Let $\frk{M}$ be a metric space whose universe is $(\mathbb{N} \times \mathbb{Z} \times {0,1})\cup(\mathbb{N}\times \{0\}\times \{2\})$ with 
\begin{itemize}
\item $d((a,n,i),(b,m,j))=1$ if $a\neq b$ or if $a=b$ and $|n-m| > 1$,
\item $d((a,n,i),(a,n+1,j)) = \frac{1}{2}$,
\item $d((a,n,0),(a,n,1)) =  \frac{1}{4} + \frac{1}{8}\cos(n)$ if $n\neq 0$,
\item $d((a,0,0),(a,0,1)) = \frac{1}{4} + \frac{1}{8}$,
\item $d((a,0,0),(a,0,2)) = \frac{1}{4} + \frac{1}{8} + 2^{-a-4}$, and
\item $d((a,0,1),(a,0,2)) = 2^{-a-4}$.
\end{itemize}
Let $T = \mathrm{Th}(\frk{M})$.
\end{desc}

\begin{ex} \label{ex:GH-GH}
A metric space theory that is strictly $\mathrm{GH}$-$\omega$-categorical and strictly $\mathrm{GH}$-$\omega_1$-categorical.
\end{ex}
\begin{desc}
Let $\frk{M}$ be a metric space whose universe is $(\mathbb{N} \times \mathbb{Z} \times {0,1})\cup(\mathbb{N}\times \{0\}\times \{2\})$ with 
\begin{itemize}
\item $d((a,n,i),(b,m,j))=1$ if $a\neq b$ or if $a=b$ and $|n-m| > 1$,
\item $d((a,n,i),(a,n+1,j)) = \frac{1}{2}$,
\item $d((a,n,0),(a,n,1)) =  \frac{1}{8} + \frac{1}{8}\cos(n)$ if $n\neq 0$,
\item $d((a,0,0),(a,0,1)) = \frac{1}{8} + \frac{1}{8}$,
\item $d((a,0,0),(a,0,2)) = \frac{1}{8} + \frac{1}{8} + 2^{-a-4}$, and
\item $d((a,0,1),(a,0,2)) = 2^{-a-4}$.
\end{itemize}
Let $T = \mathrm{Th}(\frk{M})$.
\end{desc}

If we try to use this construction to fill in the missing $\mathrm{Lip}$-$\omega$-categorical and $\mathrm{GH}$-$\omega_1$-categorical square by taking Example \ref{ex:NA-GH} and adding encoded constants in the style of Example \ref{ex:Lip-Lip}, what we get is a theory that is only $\mathrm{GH}$\nobreakdash-$\omega$\nobreakdash-categorical.

Note that these last 5 are also examples showing that a $\Delta$-$\omega_1$-categorical theory need not be unidimensional, since they contain many orthogonal types. It's even possible to modify this idea to get a discrete theory in a stratified language that is approximately uncountably categorical and yet not unidimensional (rather than using an irrational rotation of the circle along $\mathbb{Z}$-chains, use an `irrational rotation' of the $2$-adic integers along $\mathbb{Z}$-chains, or, roughly equivalently, the structure $\mathbb{N}$ together with predicates $U_i$ for each $i<\omega$ such that $U_i(n)$ is true if and only if the $i$th binary digit of $n$ is $1$), but curiously these examples seem to be limited to having trivial geometry in their types. There are also examples of strictly approximately uncountably categorical discrete theories with non-trivial geometries (such as the theory of the vector space $\mathbb{F}_p^\omega$ together with a sequence of predicates encoding projections onto the first $\omega$ $\mathbb{F}_p$ factors, with the obvious stratification), but these seem to be unidimensional. This suggests a question.

\begin{quest} \label{quest:uni}
If $T$ is a discrete theory in a stratified language $\Lcal$ which is approximately uncountably categorical and $T$ has minimal types with non-trivial geometry, does it follow that $T$ is unidimensional?
\end{quest}

Another natural question, which is related to Questions \ref{quest:disc-strat} and \ref{quest:uni}, arises from the observation that all of these examples are superstable.

\begin{quest}
If $T$ is (weakly) $\Delta$-$\kappa$-categorical for some uncountable $\kappa$, does it follow that $T$ is superstable? 
\end{quest}

\bibliographystyle{plain}
\bibliography{../ref}

\vspace{1em}

\end{document}